\documentclass[12pt]{amsart}

\usepackage{amssymb,amsthm,amsmath,url,relsize,etoolbox,cite,nameref,enumerate,fancyhdr}
\usepackage{mathtools,xcolor,enumitem,bm,bbm}
\usepackage[utf8]{inputenc}
\usepackage[hidelinks]{hyperref}
\usepackage{mathrsfs}
\usepackage{mnsymbol}
\usepackage{tabularx}
\usepackage{nicematrix}
\theoremstyle{plain}
\newtheorem*{theorem*}{Theorem}
\newtheorem{theorem}[subsection]{Theorem}

\newtheorem{lemma}[subsection]{Lemma}

\newtheorem{definition}[subsection]{Definition}

\newtheorem{remark}[subsection]{Remark}

\DeclareMathOperator{\degree}{deg}
\DeclareMathOperator{\modulus}{mod}

\DeclareMathOperator{\rad}{rad}
\DeclareMathOperator{\rank}{rank}

\DeclareMathOperator{\kernel}{ker}

\renewcommand{\Re}{\operatorname{Re}}

\newcommand*\conj[1]{\overline{#1}}

\newcommand*\mathbfe[1]{\boldsymbol{#1}}
\newcommand{\rhopi}{$(\rho , \pi )$}

\newcommand*\strictmathcal[1]{\prescript{}{\mathrm{S}}{\mathcal{#1}}}
\newcommand{\strictrho}{\rho_{\mathrm{S}}}
\newcommand{\strictpi}{\pi_{\mathrm{S}}}

\newcommand{\EvenInd}{\delta_{\mathrm{E}}} %EvenInd stands for indicator of the evens

\makeatletter
\def\l@subsection{\@tocline{2}{0pt}{2.5pc}{5pc}{}}
\makeatother

\makeatletter
\renewcommand*\env@matrix[1][*\c@MaxMatrixCols c]{%
  \hskip -\arraycolsep
  \let\@ifnextchar\new@ifnextchar
  \array{#1}}
\makeatother

\begin{document}

\title{Lattice Point Variance in Thin Elliptic Annuli over $\mathbb{F}_q [T]$}
\author{Michael Yiasemides}
\date{\today}
\address{School of Mathematical Sciences, University of Nottingham, Nottingham, NG7 2RD, UK}
\email{m.yiasemides.math@gmail.com}
%\email{michael.yiasemides@gmail.com}

\maketitle

\begin{abstract}
For fixed coprime polynomials $U,V \in \mathbb{F}_q [T]$ with degrees of different parities, and a general polynomial $A \in \mathbb{F}_q [T]$, define the arithmetic function $S_{U,V} (A)$ to be the number of representations of $A$ of the form $UE^2 + VF^2$ with $E,F \in \mathbb{F}_q [T]$. We study the mean and variance of $S_{U,V}$ over short intervals in $\mathbb{F}_q [T]$, and this can be interpreted as the function field analogue of the mean and variance of lattice points in thin elliptic annuli, where the scaling factor of the ellipses is rational. Our main result is an asymptotic formula for the variance even when the length of the interval remains constant relative to the absolute value of the centre of the interval. In terms of lattice points, this means we obtain the variance in the so-called ``local'' or ``microscopic'' regime, where the area of the annulus remains constant relative to the inner radius. We also obtain asymptotic or exact formulas for almost all other lengths of the interval, and we see some interesting behaviour at the boundary between short and long intervals. Our approach is that of additive characters and Hankel matrices that we employed for the divisor function and a restricted sum-of-squares function in previous work, and we develop further results on Hankel matrices in this paper.
\end{abstract}

%\maketitle
%\thispagestyle{fancy}
%\pagenumbering{gobble}

\allowdisplaybreaks

%\tableofcontents

\section{Introduction} \label{section, introduction}

Lattice points have been a key area of study in mathematics, due to their intrinsic interest as well as the connections to number theory \cite{Heath-Brown1992_DistrMomErrTermDirichletDivProb, Ivic2009_DivFuncRZFShortInterv} and the applications to physics \cite{BerryTabor1977_LevelClusterRegularSpectrum, BleherLebowitz1994_EnergyLevelStatModelQuantSystUnivScalingLatticePoint, BleherLebowitz1994_VarNumbLatticePointRandomNarrowElliptStrip}. A very well known problem is the Gauss circle problem, which asks for the number of lattice points (in $\mathbb{Z}^2$) that lie in a circle:
\begin{align*}
N (t)
:= \Big\lvert \big\{ (x,y) \in \mathbb{Z}^2 : x^2 + y^2 \leq t \big\} \Big\rvert .
\end{align*}
We are interested in the asymptotic behaviour as $t \longrightarrow \infty$. It can be shown that the asymptotic main term is the area of the circle, $\pi t$. The error term is dependent on the behaviour of the lattice points near the boundary of the circle, and so we can see that the error term is at most proportional to the boundary length. That is, the error is $O (t^{\frac{1}{2}})$. Although, stronger bounds exist and it is conjectured that the error is $O (t^{\frac{1}{4} +\epsilon})$. \\

A related problem asks about the number of lattice points in a circular annulus
\begin{align*}
N \big( t , a(t) \big)
:= \Big\lvert \big\{ (x,y) \in \mathbb{Z}^2 :  t \leq x^2 + y^2 \leq t + a(t) \big\} \Big\rvert .
\end{align*}
Here, we are interested in the asymptotic behaviour as $t \longrightarrow \infty$ and $a(t) = o(t)$. We note that the area of the annulus is $\pi a(t)$, while the boundary is $\asymp t^{\frac{1}{2}}$. In particular, the lattice points near the boundary will have a larger relative contribution to $N \big( t , a(t) \big)$ if $a(t)$ grows slowly, and indeed the main term may not come from the area of the annulus if $a(t)$ grows slowly enough. This makes lattice points in circular annuli an interesting problem to study. \\

We briefly remark that the set up we have established for lattice points above (and below) is not standard, in that $t$ would usually denote the radius of the circle, and we would work with the width of the annulus instead of the area $a(t)$. However, this is the setup used in \cite{BleherLebowitz1994_EnergyLevelStatModelQuantSystUnivScalingLatticePoint, BleherLebowitz1994_VarNumbLatticePointRandomNarrowElliptStrip}, and we have chosen to adopt this because we are particularly interested in the relation between our results in function fields and their classical results. \\

Now, it is worth noting that lattice points in circular annuli are directly related to the arithmetic function
\begin{align*}
r (n)
:= \sum_{\substack{x,y \in \mathbb{Z} \\ x^2 + y^2 = n }} 1
= 4 \sum_{\substack{d \mid n \\ d=1,3 (\modulus 4)}} (-1)^{\frac{d-1}{2}} ,
\end{align*}
and its behaviour over intervals. Because $\frac{1}{4} r(n)$ is multiplicative, one can associate an $L$-function to it, and, by making use of Poisson summation, one can obtain results on its distribution over intervals. See \cite{Heath-Brown1992_DistrMomErrTermDirichletDivProb, Ivic2009_DivFuncRZFShortInterv} for further details. \\

This number theoretic approach is limited when one considers the generalisation of our lattice point problems to elliptic annuli, which no longer have an associated \emph{multiplicative} arithmetic function. Thus, lattice points in elliptic annuli are particularly interesting to study. To this end, we make the following definition, again adopting the setup from \cite{BleherLebowitz1994_EnergyLevelStatModelQuantSystUnivScalingLatticePoint, BleherLebowitz1994_VarNumbLatticePointRandomNarrowElliptStrip}. For $\mu > 0$,
\begin{align*}
%\label{statement, def, lattice point in elliptic annuli notation}
N_{\mu} \big( t , a(t) \big)
:= \Big\lvert \big\{ (x,y) \in \mathbb{Z}^2 : t \leq x^2 + \mu y^2 \leq t + a(t) \big\} \Big\rvert .
\end{align*}

There are four regimes to consider, which depend on how the area of the annulus grows with respect to the boundary.

\begin{enumerate}
\item \underline{``Global'' or ``macroscopic'' regime:} $a(t)t^{-\frac{1}{2}} \longrightarrow \infty$ as $t \longrightarrow \infty$ (with $a(t) = o (t)$). In this regime, both the area and boundary tend to infinity, but the area grows at a faster rate. \\

\item \underline{``Saturation'' regime:} $a(t)t^{-\frac{1}{2}} \longrightarrow c$ as $t \longrightarrow \infty$, for a positive constant $c$. In this regime, the area and boundary both tend to infinity with the same order of growth. \\

\item \underline{``Intermediate'' or ``mesoscopic'' regime:} $a(t)t^{-\frac{1}{2}} \longrightarrow 0$ but $a(t) \longrightarrow \infty$ as $t \longrightarrow \infty$. In this regime, the area and boundary both tend to infinity, but the boundary grows at a faster rate. \\

\item \underline{``Local'' or ``microscopic'' regime:} $a(t)$ remains constant as $t \longrightarrow \infty$. In this regime, the area remains constant, while the boundary tends to infinity. \\
\end{enumerate}

Let us now highlight some of the key results that have been established in these regimes. In \cite{BleherLebowitz1994_EnergyLevelStatModelQuantSystUnivScalingLatticePoint}, Bleher and Lebowitz consider the variance
\begin{align*}
\frac{1}{T} \int_{t=T}^{2T} \Big( N_{1} \big( u^2 , a \big) - a \Big)^2 \phi (u) \mathrm{d} u ,
\end{align*}
where $\phi (t)$ is a suitable smoothing function given in (1.12) of \cite{BleherLebowitz1994_EnergyLevelStatModelQuantSystUnivScalingLatticePoint}. Note the replacement of $t$ by $u^2$; that $a$ is taken to be constant within the integral (although after the integral is performed we will consider limiting behaviour for $a$); and that $\mu = 1$ (that is, we are considering circular annuli). For the global regime, they prove that
\begin{align*}
\lim_{\substack{aT^{-2} \rightarrow 0 \\ a T^{-1} \rightarrow \infty }}
	\frac{1}{T} \int_{t=T}^{2T} \Big( N_{1} \big( u^2 , a \big) - \pi a \Big)^2 \phi (u) \mathrm{d} u
= v T ,
\end{align*}
for a constant $v$. For the saturation regime, they prove that
\begin{align*}
\lim_{a T^{-1} \rightarrow c }
	\frac{1}{T} \int_{t=T}^{2T} \Big( N_{1} \big( u^2 , a \big) - \pi a \Big)^2 \phi (u) \mathrm{d} u
= V (c) T ,
\end{align*}
for a function $V$ that they compute as an infinite series. They show that 
\begin{align*}
\lim_{T \longrightarrow \infty} \frac{1}{T} \int_{z=0}^{T} V(z) \mathrm{d} z
= v . 
\end{align*}
For the intermediate regime, for any constant $\omega \in \Big( 0 , \frac{1}{2} \Big)$, Bleher and Lebowitz \cite{BleherLebowitz1994_VarNumbLatticePointRandomNarrowElliptStrip} prove
\begin{align*}
\lim_{\substack{aT^{-\frac{1}{2}} \rightarrow 0 \\ a T^{-\omega} \rightarrow \infty }}
	\frac{1}{T} \int_{t=T}^{2T} \frac{\Big( N_{\mu} \big( t , a \big) - \pi \mu^{-\frac{1}{2}} a \Big)^2}{\pi \mu^{-\frac{1}{2}} a} \mathrm{d} t
= \begin{cases}
4 &\text{ if $\mu$ is Diophantine,} \\
\infty &\text{ if $\mu$ is a Liouville number.}
\end{cases}
\end{align*}
Whereas, if $\mu$ is rational and is equal to $\frac{p}{q}$ for coprime $p,q$, they prove
\begin{align} \label{statement, rational ellipse variance, intermediate regime}
\lim_{\substack{aT^{-\frac{1}{2}} \rightarrow 0 \\ a T^{-\omega} \rightarrow \infty }}
	\frac{1}{T} \int_{t=T}^{2T} \frac{\Big( N_{\mu} \big( t , a \big) - \pi \mu^{-\frac{1}{2}} a \Big)^2}{\pi \mu^{-\frac{1}{2}} a \lvert \log d \rvert} \mathrm{d} t
= c (\mu) ;
\end{align}
where $d = 2\pi \mu^{-\frac{1}{2}} a T^{-\frac{1}{2}}$, and
\begin{align*}
c (\mu)
= \begin{cases}
\frac{4}{\pi} d(p) d(q) (pq)^{-\frac{1}{2}} &\text{ if $p,q \equiv 1 \modulus 2$,} \\
& \\
\frac{1}{\pi} (6l+2) d(p') d(q) (pq)^{-\frac{1}{2}} &\text{ if $p = 2^l p'$ and $p',q \equiv 1 \modulus 2$.} 
\end{cases}
\end{align*}
It may be helpful for the reader to keep in mind that for the results from \cite{BleherLebowitz1994_VarNumbLatticePointRandomNarrowElliptStrip}, Bleher and Lebowitz use the variable $t$, whereas in \cite{BleherLebowitz1994_EnergyLevelStatModelQuantSystUnivScalingLatticePoint} they use $u^2$; and this effects the limits one must take for each regime. \\

Finally, again in the intermediate regime, assuming the additional condition that for all $\delta < 0$ we have $T,a \longrightarrow \infty$ with $\frac{a}{T} \longrightarrow 0$ and $\frac{a}{T} \gg T^{-\delta}$, Hughes and Rudnick \cite{HughesRudnick2004_DistrLatticePointThinAnnuli} prove for any interval $[a,b]$ that 
\begin{align*}
\lim_{T \longrightarrow \infty} \frac{1}{T} \mathrm{meas} 
	\bigg\{ t \in [T,2T] : \frac{N_1 (u^2 , a) - \pi a}{\sigma_1 u^{\frac{1}{2}}} \in [a,b] \bigg\}
= \frac{1}{\sqrt{2 \pi}} \int_{x=a}^{b} e^{-\frac{x^2}{2}} \mathrm{d} x ,
\end{align*}
where
\begin{align*}
{\sigma_1}^2
:= 16 \frac{a}{u} \log \big( \frac{u}{a} \big) .
\end{align*}
That is, we have a Gaussian limiting distribution. Wigman \cite{Wigman2006_DistrLattPointElliptAnnuli} extends upon this, to the elliptic cases for which $\mu^{\frac{1}{2}}$ is transcendental and strongly Diophantine\footnote{``Strongly Diophantine'' is defined by Wigman in \cite{Wigman2006_DistrLattPointElliptAnnuli}}, by proving that 
\begin{align*}
\lim_{T \longrightarrow \infty} \frac{1}{T} \mathrm{meas} 
	\bigg\{ t \in [T,2T] : \frac{N_1 (u^2 , a) - \pi \mu^{-\frac{1}{2}} a}{\sigma_2 u^{\frac{1}{2}}} \in [a,b] \bigg\}
= \frac{1}{\sqrt{2 \pi}} \int_{x=a}^{b} e^{-\frac{x^2}{2}} \mathrm{d} x ,
\end{align*}
where
\begin{align*}
{\sigma_2}^2
:= \frac{8 \pi}{\mu^{\frac{1}{2}}} a u^{-1} .
\end{align*}

Let us now discuss the results we prove in this paper. First, we will need to introduce some notation. Let $\mathcal{A} := \mathbb{F}_q [T]$, the polynomial ring over the finite field of order $q$, where $q$ will be taken to be an odd prime power. For $A \in \mathcal{A} \backslash \{ 0 \}$, we define $\lvert A \rvert := q^{\degree A}$, and $\lvert 0 \rvert := 0$. We denote the set of monics by $\mathcal{M}$, and the set of monic primes by $\mathcal{P}$. For $\mathcal{S} \subseteq \mathcal{A}$ and integers $n \geq 0$, we define $\mathcal{S}_n , \mathcal{S}_{\leq n} , \mathcal{S}_{<n}$ to be the set of elements in $\mathcal{S}$ with degree $n , \leq n , < n$, respectively (we take $\degree 0$ to be $-\infty$). For $A \in \mathcal{A}$ and $h \geq 0$, we define the interval of centre $A$ and radius $h$ by $I(A;<h) := \{ B \in \mathcal{A} : \degree (B-A) < h \}$.\footnote{In \cite{KeatingRodgersRudnik2018_SumDivFuncFqtMatrInt} they use a slightly different notation, namely $I(A;h) := \{ B \in \mathcal{A} : \degree (B-A) \leq h \}$. We prefer to use our definition of $I(A;<h)$ because it is more natural in some regards, in that $\lvert I(A;<h) \rvert = q^h$ as opposed to $\lvert I(A;h) \rvert = q^{h+1}$. Due to the explicit difference in notation, this should not cause any confusion.} We denote the greatest common (monic) divisor of two polynomials $A,B$, not both zero, by $(A,B)$. For a monic polynomial $A \in \mathcal{A} \backslash \{ 0 \}$, we define the totient function by
\begin{align*}
\phi (A)
:= \sum_{\substack{C \in \mathcal{A} \\ \degree C < \degree A \\ (C,A)=1 }}1 .
\end{align*}
Furthermore, we define the radical of a polynomial $A \neq 0$ by $\rad A := \prod_{P \mid A} P$, where the product is over \emph{primes}. Now let us define our function that counts the number of ``elliptic'' representations of a polynomial.

\begin{definition}
Let $U,V \in \mathcal{M}$ be coprime and such that $\degree U$ is even and $\degree V$. For $B \in \mathcal{A}$, we define
\begin{align*}
S_{U,V} (B)
:= \lvert \{ (E,F) \in \mathcal{A}^2 : B = UE^2 + VF^2 \} \rvert .
\end{align*}
\end{definition}

\begin{remark}\label{remark, explanation why U,V are odd, even coprime monic}.
Taking the degrees of $U,V$ to have different parities seems to be the natural choice over $\mathbb{F}_q [T]$. Indeed, if both were even/odd then $S_{U,V} (B)$ would be zero for all $B$ with odd/even degree. Furthermore,
if $U,V$ were both even or both odd, then for certain cases of $q$ (the order of our finite field) we could have cancellation when we sum $UE^2$ and $VF^2$. Taking $U,V$ to have degrees of different parities avoids all of these problems. \\

Suppose $B$ is monic. We note that if $\degree B = n$ is even, then we must have that $\degree E = \frac{n - \degree U}{2}$ and that $E$ or $-E$ are monic, while $F$ takes values in $\mathcal{A}$ with $\degree F \leq \frac{n - \degree V -1}{2}$. On the other hand, if $\degree B = n$ is odd, then we must have that $\degree F = \frac{n - \degree V}{2}$ and that $F$ or $-F$ are monic, while $E$ takes values in $\mathcal{A}$ with $\degree E \leq \frac{n-\degree U -1}{2}$.
\end{remark}

Now, we wish to understand the behaviour of $S_{U,V} (B)$ over intervals. For the mean, we have
\begin{align}
\begin{split} \label{statement, lattice point ellipse mean value calculations}
\frac{1}{q^n} \sum_{A \in \mathcal{M}_n } \sum_{B \in I (A; <h)} S_{U,V} (B)
= &\frac{q^h}{q^n} \sum_{A \in \mathcal{M}_n} S_{U,V} (A)
= \frac{q^h}{q^n} \sum_{A \in \mathcal{M}_n} \sum_{\substack{E,F \in \mathcal{A} : \\ A = UE^2 + VF^2}} 1 \\
= &\begin{cases}
\frac{2q^h}{q^n} \sum_{\substack{E \in \mathcal{M} \\ \degree E = \frac{n-\degree U}{2}}} \sum_{\substack{F \in \mathcal{A} \\ \degree F \leq \frac{n- \degree V -1}{2}}} 1 &\text{ if $n$ is even} \\
\frac{2q^h}{q^n} \sum_{\substack{E \in \mathcal{A} \\ \degree E = \frac{n-\degree U-1}{2}}} \sum_{\substack{F \in \mathcal{M} \\ \degree F \leq \frac{n- \degree V }{2}}} 1 &\text{ if $n$ is odd}
\end{cases} \\
= &2 q^{h - \frac{\degree U}{2} - \frac{\degree V}{2} + \frac{1}{2}} .
\end{split}
\end{align}

We now define
\begin{align*}
\Delta_{S_{U,V}} (A; <h)
= \sum_{B \in I (A; <h)} S_{U,V} (B)
	\hspace{1em} - q^{h - \frac{\degree U}{2} - \frac{\degree V}{2} + \frac{1}{2}} ,
\end{align*}
and we are interested in the variance
\begin{align*}
\frac{1}{q^n} \sum_{A \in \mathcal{M}_n } \bigg\lvert \Delta_{S_{U,V}} (A;<h) \bigg\rvert^2 .
\end{align*}
Let us describe the analogies for $\mathbb{F}_q [T]$ of the four regimes. We have that $q^n$ is analogous to $t$, while $q^h$ is analogous to the area $a$. In what follows, we always have $h \leq n$.

\begin{enumerate}
\item \underline{Global regime:} $n,h \longrightarrow \infty$ with $ h - \frac{n}{2} \longrightarrow \infty$. \\

\item \underline{Saturation regime:} $n,h \longrightarrow \infty$ with $h - \frac{n}{2} \longrightarrow c$, for a real constant $c$. \\

\item \underline{Intermediate regime:} $n,h \longrightarrow \infty$ with $\frac{n}{2} - h \longrightarrow \infty$. \\

\item \underline{Local regime:} $n \longrightarrow \infty$ with $h$ constant. \\
\end{enumerate}

Our main result is the following theorem. There are several cases due to the various ranges of $h$ that are covered, but the most significant is (\ref{statement, main theorem, case 3 asymptotic}) as this addresses short intervals, and in the context of lattice points it includes the intermediate and local regimes.

\begin{theorem} \label{main theorem, lattice point variance elliptic annuli}
Let $U,V \in \mathcal{M}$ with $\degree U$ even and $\degree V$ odd, and $(U,V)=1$. Let $0 \leq h \leq n$. In the interest of presentation, we define
\begin{align*}
&s:= \begin{cases} \degree U &\text{ if $n$ is even,} \\ \degree U +1 &\text{ if $n$ is odd;} \end{cases} 
&&s' := \frac{n-s}{2} ;
&&&n_1 := \Big\lfloor \frac{n+2}{2} \Big\rfloor ; \\
&t = \begin{cases} \degree V +1 &\text{ if $n$ is even,} \\ \degree V &\text{ if $n$ is odd;} \end{cases} 
&&t' := \frac{n-t}{2} ;
&&&n_2 := \Big\lfloor \frac{n+3}{2} \Big\rfloor .
\end{align*}
We have the following cases. \\

Case 1: If $n$ is even and $h \geq s'+s$, or if $n$ is odd and $h \geq t'+t$, then
\begin{align*}
\frac{1}{q^n} \sum_{A \in \mathcal{M}_n } \bigg\lvert \Delta_{S_{U,V}} (A;<h) \bigg\rvert^2 
= 0 . \\
\end{align*}

Case 2: If $n$ is even and $n_2 -1 \leq h < s'+s$, or if $n$ is odd and $n_2 -1 \leq h < t'+t$, we have
\begin{align*}
\frac{1}{q^n} \sum_{A \in \mathcal{M}_n } \bigg\lvert \Delta_{S_{U,V}} (A;<h) \bigg\rvert^2 
= q^{h} f_{U,V} (n,h) ;
\end{align*}
where, if $n$ is even, then
\begin{align*}
f_{U,V} (n,h)
= \frac{4(q-1)}{q^{\frac{1}{2}} \lvert UV \rvert^{\frac{1}{2}} } \sum_{r_1 = s'+1}^{n-h} q^{r_1 -(n-h)}
		\sum_{\substack{B_1 \in \mathcal{A}_{\leq s'+s-r_1} \\ B_2 \in \mathcal{M}_{\leq r_1 -s' -1} \\ (B_2 , U) \mid B_1 }} \lvert (B_2 , U) \rvert
		\sum_{\substack{C_1 \in \mathcal{A}_{\leq t'+t-r_1 -1} \\ C_2 \in \mathcal{A}_{\leq r_1 -t' -2} \\ (C_2 , V) \mid C_1 }} \lvert (C_2 , V) \rvert ;
\end{align*}
and if $n$ is odd, then
\begin{align*}
f_{U,V} (n,h)
= \frac{4(q-1)}{q^{\frac{1}{2}} \lvert UV \rvert^{\frac{1}{2}} } \sum_{r_1 = t'+1}^{n-h} q^{r_1 -(n-h)}
		\sum_{\substack{B_1 \in \mathcal{A}_{\leq s'+s-r_1 -1} \\ B_2 \in \mathcal{A}_{\leq r_1 -s' -2} \\ (B_2 , U) \mid B_1 }} \lvert (B_2 , U) \rvert
		\sum_{\substack{C_1 \in \mathcal{A}_{\leq t'+t-r_1} \\ C_2 \in \mathcal{M}_{\leq r_1 -t' -1} \\ (C_2 , V) \mid C_1 }} \lvert (C_2 , V) \rvert .
\end{align*}
Furthermore, we have the following bound for all $n$:
\begin{align*}
f_{U,V} (n,h)
\leq 4 q^{\frac{1}{2}}
	\lvert UV \rvert^{\frac{1}{2}} 
	(\log_q \degree U)
	(\log_q \degree V) . \\
\end{align*}

Case 3: Suppose $3 (\degree UV +1) \leq h < \min \{ s' , t' \} -1$ (and $n$ can be even or odd). Then,
\begin{align*}
&\frac{1}{q^n} \sum_{A \in \mathcal{M}_n } \bigg\lvert \Delta_{S_{U,V}} (A;<h) \bigg\rvert^2 \\
= &4 (1-q^{-1}) q^h M (U,V) \Big( \frac{n}{2} -h \Big) 
	+ q^{2h - (n_2 -1)} f_{U,V} (n , n_1 -1) 
	+ O_{\epsilon} (q^h \lvert UV \rvert^{-1+\epsilon}) 
	+ O (q^{h + \frac{3}{2}} \degree UV ) ;
\end{align*}
where, for a non-zero polynomial $A$, we define $e_P (A)$ to be the maximal integer such that $P^{e_P (A)} \mid A$; and
\begin{align*}
M (U,V)
:= \lvert UV \rvert^{-1} 
	\prod_{P \mid UV} \Bigg( 1 + \bigg( \frac{1- \lvert P \rvert^{-1}}{1+ \lvert P \rvert^{-1}} \bigg)  e_P (UV) \Bigg) .
\end{align*}
In particular, if $n \longrightarrow \infty$ with $\frac{n}{2} - h \longrightarrow \infty$ (note this includes both the intermediate and local regimes), we have
\begin{align} \label{statement, main theorem, case 3 asymptotic}
\frac{1}{q^n} \sum_{A \in \mathcal{M}_n } \bigg\lvert \Delta_{S_{U,V}} (A;<h) \bigg\rvert^2 
\sim 4 (1-q^{-1}) q^h M (U,V) \Big( \frac{n}{2} -h \Big) .
\end{align}
\end{theorem}

It is of interest to note that (\ref{statement, main theorem, case 3 asymptotic}) bears similarity to (\ref{statement, rational ellipse variance, intermediate regime}). The scalings are different, but we can see that $\big( \frac{n}{2} -h \big)$ in (\ref{statement, main theorem, case 3 asymptotic}) accounts for the $\rvert \log d \rvert \approx \big( \frac{1}{2} \log T - \log a \big)$ in $(\ref{statement, rational ellipse variance, intermediate regime})$. Also, the factor of $M (U,V)$ in (\ref{statement, main theorem, case 3 asymptotic}) is similar to, but not the exact analogue of, $d(p) d(q)$ in (\ref{statement, main theorem, case 3 asymptotic}). It would be interesting to investigate why this difference occurs. Furthermore, we note that (\ref{statement, main theorem, case 3 asymptotic}) includes the local regime, and as far as we are aware this is the first time that results on this regime have been obtained, either in the classical or function field setting. \\

Cases 1 and 2 include the saturation regime for $c$ positive (recall, in our description of the saturation regime, $c$ is such that $h - \frac{n}{2} \longrightarrow c$). When $c$ is very small, and thus $h$ is just slightly larger than $\frac{n}{2}$, this is case 2 and we see some interesting behaviour. \\

Finally, the global regime is also included in Case 1, and we see vanishing of the variance. This is not unusual in function fields; for example, we see vanishing of the variance over large intervals of the divisor function as well. \\

The approach we employ to prove Theorem \ref{main theorem, lattice point variance elliptic annuli} is different to the approaches that have been used for the classical problem described earlier, and our approach makes use of the structure of polynomial ring. \\

Full details can be found in Section \ref{section, Hankel matrices}, but a short summary of our approach is the following. We count solutions to the equation $B - UE^2 - VF^2 = 0$ using the indicator function
\begin{align*}
\text{$\mathbbm{1} (A)$}
:= \begin{cases}
1 &\text{ if $A=0$,} \\
0 &\text{ if $A \in \mathcal{A} \backslash \{ 0 \}$.}
\end{cases}
\end{align*}
We then take a certain Fourier expansion of $\mathbbm{1}$ using additive characters and, because each function in the expansion is additive, we can consider each term $B$, $UE^2$, and $VF^2$ separately. This ultimately requires us to understand the kernel structure of Hankel matrices\footnote{A Hankel matrix is a matrix where all the entries on a given skew-diagonal are identical.} over $\mathbb{F}_q$. Interestingly, the kernel of a Hankel matrix can be interpreted as the linear span of two coprime polynomials. Details can be found in Section \ref{section, Hankel matrices}, but a brief indication of how Hankel matrices appear is the following. Suppose we have two polynomials $E = e_0 + e_1 T + \ldots + e_l T^l \in \mathcal{A}_l$ and $F = f_0 + f_1 T + \ldots + f_m T^m \in \mathcal{A}_m$ with $l+m=n$. With a few calculations, we can see that
\begin{align*}
\begin{matrix}
EF = \begin{pmatrix} e_0 & e_1 & \cdots & e_l \end{pmatrix} \\ \\ \\ \\ \\ \\
\end{matrix}
\begin{pmatrix}
1 & T & T^2 & \cdots & \cdots & T^m \\
T & T^2 &  &  &  & \vdots \\
T^2 &  &  &  &  & \vdots \\
\vdots &  &  &  &  & \vdots \\
\vdots &  &  &  &  & T^{n-1} \\
T^l & \cdots & \cdots & \cdots & T^{n-1} & T^{n} 
\end{pmatrix}
\begin{matrix}
\begin{pmatrix}
f_0 \\ f_1 \\ \vdots \\ f_m
\end{pmatrix}
\\ \\ \\
\end{matrix} .
\end{align*}

Now, this is the approach we previously used for the variance over short intervals in $\mathbb{F}_q [T]$ for the divisor function \cite{Yiasemides2021_VariCorrDivFuncFpTHankelMatr_ArXiv_v2} and for a restricted sum-of-squares function \cite{Yiasemides2022_VariSumTwoSquareOverIntervalFqT_I_Arxiv}. Indeed, in those cases we are counting solutions to the equations $B-EF=0$ and $B-  E^2 - F^2 = 0$. However, in this paper the situation is more complex due to the fact that we have terms in the equation $B - UE^2 - VF^2 = 0$ that are products of three polynomials. Typically, this would mean we must work with three-dimensional Hankel arrays (tensors), which is considerably more difficult. Indeed, this is also the obstacle that we would encounter were we to investigate the variance over short intervals of higher divisor functions, such as $d_3$ where we would work with the equation $B - EFG = 0$; and this is an unsolved problem with powerful implications. However, because $U,V$ are fixed, we are able to make progress, but we must obtain further results on Hankel matrices. Specifically, these are Lemmas \ref{lemma, bijection from quasiregular Hankel matrices to M_r times A_<r} and \ref{lemma, U reduction, 1 dimensional kernel case}, and they involve understanding the common factors between $UV$ and the polynomials found in the kernels of Hankel matrices. Most of the rest of Section \ref{section, Hankel matrices} is dedicated to providing results on Hankel matrices that we have already established in \cite {Yiasemides2021_VariCorrDivFuncFpTHankelMatr_ArXiv_v2, Yiasemides2022_VariSumTwoSquareOverIntervalFqT_I_Arxiv}. \\

The proof of the entirety of Theorem \ref{main theorem, lattice point variance elliptic annuli} is given in Section \ref{section, main theorem proof}. For a discussion on how one could investigate the extension of our approach to (1) moments higher than the variance, (2) to other arithmetic functions such as higher divisor functions, and (3) to a possible classical analogue, then we refer the reader to Subsection 1.4 of \cite{Yiasemides2021_VariCorrDivFuncFpTHankelMatr_ArXiv_v2}.

\section{Hankel Matrices} \label{section, Hankel matrices}

An $(l+1) \times (m+1)$ Hankel matrix over $\mathbb{F}_q$ is a matrix of the form 
\begin{align*}
\begin{pNiceMatrix}
\alpha_0 & \alpha_1 & \alpha_2 & \Cdots &  & \alpha_m \\
\alpha_1 & \alpha_2 &  &  &  & \Vdots \\
\alpha_2 &  &  &  &  & \\
\Vdots &  &  &  &  & \\
 &  &  &  &  & \alpha_{n-1} \\
\alpha_l &  &  & \Cdots & \alpha_{n-1} & \alpha_{n}
\end{pNiceMatrix} ,
\end{align*}
where $n := l+m$ and $\alpha_0 , \ldots , \alpha_n \in \mathbb{F}_q$. This is similar to a Toeplitz matrix, but in this case all entries on a given skew-diagonal are the same. We can associate the sequence
\begin{align*}
\mathbfe{\alpha}
:= (\alpha_0 , \alpha_1 , \ldots , \alpha_n)
\end{align*}
to the matrix above, and denote the matrix by $H_{l+1 , m+1} (\mathbfe{\alpha})$. We can see that there are $n+1$ possibly values for $l,m \geq 0$ such that $l+m=n$, and so there are $n+1$ possible matrices we can associate with $\mathbfe{\alpha}$ in this way. As we will briefly discuss later, when studying the kernel of Hankel matrices, it is natural to group the matrices that have the same associated sequence and consider how the kernel differs between $H_{l+1 , m+1} (\mathbfe{\alpha})$ and $H_{(l+1)-1 , (m+1)+1} (\mathbfe{\alpha}) = H_{l , m+2} (\mathbfe{\alpha})$. This then allows us to understand the kernel structure of Hankel matrices generally. \\

Hankel matrices are defined by the characteristic that $\alpha_{i,j} = \alpha_{k,l}$ whenever $i+j=k+l$, where $\alpha_{i,j}$ is the $(i,j)$-th entry of the matrix. The can be generalised to higher dimensional arrays (tensors). For example, if we have a $k$ dimensional tensor with $(i_1 , \ldots , i_k)$-th entry denoted by $\alpha_{i_1 , \ldots , i_k}$, then we say it is Hankel if $\alpha_{i_1 , \ldots , i_k} = \alpha_{j_1 , \ldots , j_k}$ whenever $i_1 + \ldots i_k = j_1 + \ldots + j_k$. Note that one-dimensional arrays (that is, vectors) are always Hankel. \\

At this point, one may ask what the connection between Hankel matrices and function fields is. To answer this, let us first make a definition. Suppose we have a polynomial $B=b_0 + b_1 T + \ldots + b_n T^n \in \mathbb{F}_q [T]$ with $b_n \neq 0$. Then, for $k \geq n$, we define the vector $[B]_k \in \mathbb{F}_q^{k+1}$ by
\begin{align*}
[B]_k
:= \begin{pmatrix}
b_0 \\ b_1 \\ \vdots \\ b_n \\ 0 \\ \vdots \\ 0
\end{pmatrix} ,
\end{align*}
where there are $k-n$ zeros at the end. Typically, we will have $k=n$ or $k=n+1$ (and so either no zeros or one zero at the end), but not always. Also, if $B=0$ then $[B]_k$ is just defined to be the $(k+1) \times 1$ zero vector. In essence, $[B]_k$ is just a representation of $B$ as a vector, but we are also keeping track of the number of zeros at the end. Now, suppose that $B \in \mathcal{A}_n$. We clearly have
\begin{align} \label{statement, B as Hankel vector product}
\begin{matrix}
B = \begin{pmatrix} b_0 & b_1 & \cdots & b_n \end{pmatrix} \\ \\ \\ \\
\end{matrix}
	\begin{pmatrix} 1 \\ T \\ \vdots \\ T^n \end{pmatrix} .
\end{align}
Now suppose we have $E = e_0 + e_1 T + \ldots + e_l T^l \in \mathcal{A}_l$ and $F = f_0 + f_1 T + \ldots + f_m T^m \in \mathcal{A}_m$ with $l+m=n$. With a few calculations, we can see that
\begin{align} \label{statement, EF as Hankel vector product}
\begin{matrix}
EF = \begin{pmatrix} e_0 & e_1 & \cdots & e_l \end{pmatrix} \\ \\ \\ \\ \\ \\
\end{matrix}
\begin{pmatrix}
1 & T & T^2 & \cdots & \cdots & T^m \\
T & T^2 &  &  &  & \vdots \\
T^2 &  &  &  &  & \vdots \\
\vdots &  &  &  &  & \vdots \\
\vdots &  &  &  &  & T^{n-1} \\
T^l & \cdots & \cdots & \cdots & T^{n-1} & T^{n} 
\end{pmatrix}
\begin{matrix}
\begin{pmatrix}
f_0 \\ f_1 \\ \vdots \\ f_m
\end{pmatrix}
\\ \\ \\
\end{matrix} .
\end{align}
What we see here is ``Hankel structure'' appearing in polynomial multiplication. That is, in (\ref{statement, B as Hankel vector product}), we have a product of one polynomial (just $B$), and this involves the Hankel vector $(1 , T , \cdots , T^n)^T$. In (\ref{statement, EF as Hankel vector product}), we have the product of two polynomials $E,F$, and this involves the $(l+1) + (m+1)$ Hankel matrix above. If we were to consider a product of three polynomials, we could obtain a similar representation as above but with a three-dimensional Hankel tensor. \\

This gives an indication of how ``Hankel structure'' relates to polynomial multiplication. However, to see how Hankel matrices appear, that are specifically over $\mathbb{F}_q$, we will first need to make the following two definitions.

\begin{definition}[Additive Character]
A function $\psi : \mathbb{F}_q \longrightarrow \mathbb{C}^*$ is an additive character on $\mathbb{F}_q$ if $\psi (a+b) = \psi (a) \psi (b)$ for all $a,b \in \mathbb{F}_q$. We say it is non-trivial if there exists some $c \in \mathbb{F}_q^*$ such that $\psi (c) \neq 1$. In the remainder of the paper, we will take $\psi$ to be a fixed non-trivial additive character on $\mathbb{F}_q$, unless otherwise stated.
\end{definition}

Non-trivial additive characters satisfy the orthogonality relation
\begin{align} \label{statement, additive character FF orthog relation}
\frac{1}{q} \sum_{a \in \mathbb{F}_q} \psi (ab)
= \begin{cases}
1 &\text{ if $b=0$,} \\
0 &\text{ if $b \in \mathbb{F}_q^*$.}
\end{cases}
\end{align}
They can be viewed as analogous to the classical exponential function, and in a similar way that the exponential function is used in Fourier expansions, we can define Fourier expansions over $\mathcal{A}$ using additive characters. In this paper we will only require the Fourier expansion of the indicator function, for its use in counting solutions to Diophantine equations as described in Section \ref{section, introduction}.

\begin{definition}[Fourier Expansion of the Indicator Function] \label{definition, Fourier exp indicator function}
We define $\mathbbm{1} : \mathcal{A} \longrightarrow \mathbb{C}$ by
\begin{align*}
\text{$\mathbbm{1} (A)$}
= \begin{cases}
1 &\text{ if $A=0$,} \\
0 &\text{ if $A \in \mathcal{A} \backslash \{ 0 \}$;}
\end{cases}
\end{align*}
and, for any $n \geq 0$ and $A \in \mathcal{A}_{\leq n}$, we have the Fourier expansion
\begin{align*}
\mathbbm{1} (A) 
= \frac{1}{q^{n+1}} \sum_{\mathbfe{\alpha} \in \mathbb{F}_q^{n+1}} \psi (\mathbfe{\alpha} \cdot [A]_n) .
\end{align*}
\end{definition}

In some cases, we will have products $EF$ instead of $A$ above, and this is how Hankel matrices over $\mathbb{F}_q$ appear. Indeed, we have
\begin{align} \label{statement, how Hankel matrices appear from products}
\begin{matrix}
\mathbfe{\alpha} \cdot [EF]_n = \begin{pmatrix} e_0 & e_1 & \cdots & e_l \end{pmatrix} \\ \\ \\ \\ \\ \\
\end{matrix}
\begin{pmatrix}
1 & \alpha_1 & \alpha_2 & \cdots & \cdots & \alpha_m \\
\alpha_1 & \alpha_2 &  &  &  & \vdots \\
\alpha_2 &  &  &  &  & \vdots \\
\vdots &  &  &  &  & \vdots \\
\vdots &  &  &  &  & \alpha_{n-1} \\
\alpha_l & \cdots & \cdots & \cdots & \alpha_{n-1} & \alpha_{n} 
\end{pmatrix}
\begin{matrix}
\begin{pmatrix}
f_0 \\ f_1 \\ \vdots \\ f_m
\end{pmatrix}
\\ \\ \\
\end{matrix} .
\end{align}

Now that we have seen how Hankel matrices relate to function fields, let us recall some definitions and results that we have established in \cite{Yiasemides2021_VariCorrDivFuncFpTHankelMatr_ArXiv_v2, Yiasemides2022_VariSumTwoSquareOverIntervalFqT_I_Arxiv} that we will require later. \\

In the remainder of the paper, for an integer $n \geq 0$, we define $n_1 := \lfloor \frac{n+2}{2} \rfloor$ and $n_2 := \lfloor \frac{n+3}{2} \rfloor$. In particular, if we have $\mathbfe{\alpha} = (\alpha_0 , \ldots , \alpha_n)$, then $H_{n_1 , n_2} (\mathbfe{\alpha})$ is the square/almost square matrix associated to $\mathbfe{\alpha}$, depending on whether $n$ is even/odd. \\

Now suppose we have $l,m$ such that $l+m =: n' \leq n$. If we have the equality $n'=n$, then $H_{l+1,m+1} (\mathbfe{\alpha})$ is well-defined based on our definitions above. If $n' < n$, then we define
\begin{align*}
H_{l+1,m+1} (\mathbfe{\alpha})
:= H_{l+1,m+1} (\mathbfe{\alpha}')
\end{align*}
where $\mathbfe{\alpha}' := (\alpha_0 , \ldots , \alpha_{n'})$ is a truncation of $\mathbfe{\alpha}$ and the right side is well defined based on our previously established definitions. In particular, the matrices $H_{1,1} (\mathbfe{\alpha}) , \ldots , H_{n_1 , n_1} (\mathbfe{\alpha})$ are top-left submatrices of $H_{n_1 , n_2} (\mathbfe{\alpha})$. \\

We now define the \rhopi-characteristic of a Hankel matrix, which is given in \cite{HeinigRost1984_AlgMethToeplitzMatrOperat}. There, their definition is different but ultimately equivalent to ours. Ours is based more on the results found in \cite{Garcia-ArmasGhorpadeRam2011_RelativePrimePolyNonsingHankelMatrFinField}.

\begin{definition}[The \rhopi-characteristic.]
Let $\mathbfe{\alpha} \in \mathbb{F}_q^{n+1}$. 
\begin{itemize}
\item We define $r (\mathbfe{\alpha})$ to be the rank of $H_{n_1 , n_2} (\mathbfe{\alpha})$. \\

\item We define $\rho (\mathbfe{\alpha})$ to be the largest integer $\rho_1$ in $\{ 1 , \ldots , n_1 \}$ such that $H_{\rho_1 , \rho_1} (\mathbfe{\alpha})$ is invertible, and if no such $\rho_1$ exists then we take $\rho (\mathbfe{\alpha}) = 0$. \\

\item We also define $\pi (\mathbfe{\alpha}) := r (\mathbfe{\alpha}) - \rho (\mathbfe{\alpha})$. 
\end{itemize}

We say the \rhopi-characteristic of $\mathbfe{\alpha}$ is $(\pi_1 , \rho_1)$ if $\rho (\mathbfe{\alpha}) = \rho_1$ and $\pi (\mathbfe{\alpha}) = \pi_1$. 
\end{definition}

The \rhopi-characteristic is a natural property to study, as will become clear later. We will also define the strict \rhopi-characteristic, which is very similar.

\begin{definition}[The strict \rhopi-characteristic.]
Let $\mathbfe{\alpha} \in \mathbb{F}_q^{n+1}$. 
\begin{itemize}
\item We define $r (\mathbfe{\alpha})$ to be the rank of $H_{n_1 , n_2} (\mathbfe{\alpha})$. \\

\item We define $\strictrho (\mathbfe{\alpha})$ to be the largest integer $\rho_1$ in $\{ 1 , \ldots , n_2 -1 \}$ such that $H_{\rho_1 , \rho_1} (\mathbfe{\alpha})$ is invertible, and if no such $\rho_1$ exists then we take $\rho (\mathbfe{\alpha}) = 0$. \\

\item We also define $\strictpi (\mathbfe{\alpha}) := r (\mathbfe{\alpha}) - \strictrho (\mathbfe{\alpha})$. 
\end{itemize}

We say the strict \rhopi-characteristic of $\mathbfe{\alpha}$ is $(\pi_1 , \rho_1)$ if $\strictrho (\mathbfe{\alpha}) = \rho_1$ and $\strictpi (\mathbfe{\alpha}) = \pi_1$. 
\end{definition}

The only difference between the \rhopi-characteristic and the strict \rhopi-characteristic is in the definition of $\strictrho (\mathbfe{\alpha})$ (although, it does indirectly affect the definition of $\strictpi (\mathbfe{\alpha})$ as well). We have that $\rho (\mathbfe{\alpha})$ can take values up to $n_1$, while $\strictrho (\mathbfe{\alpha})$ can take values up to $n_2 -1$. Note that if $n$ is odd then $n_1 = n_2 -1$, and so the difference exists only when $n$ is even. Even then, the value of $\strictrho (\mathbfe{\alpha})$ differs from $\rho (\mathbfe{\alpha})$ only when $H_{n_1 , n_1} (\mathbfe{\alpha})$ has full rank (and thus is invertible). Studying Hankel matrices often involves viewing them as extensions of their submatrices, and it turns out that these full-rank square Hankel matrices mentioned above are ``threshold'' cases in the extensions, and this leads to two possible natural concepts for the \rhopi-characteristic.

It will be helpful at times to extend our definitions to the Hankel matrices associated to $\mathbfe{\alpha}$: For $l,m \geq 0$ with $l+m=n$, we define
\begin{align*}
&r \big( H_{l+1 , m+1} (\mathbfe{\alpha}) \big) := r (\mathbfe{\alpha}) , \\
&\rho \big( H_{l+1 , m+1} (\mathbfe{\alpha}) \big) := \rho (\mathbfe{\alpha}) , 
	&&\strictrho \big( H_{l+1 , m+1} (\mathbfe{\alpha}) \big) := \strictrho (\mathbfe{\alpha}) , \\
&\pi \big( H_{l+1 , m+1} (\mathbfe{\alpha}) \big) := \pi (\mathbfe{\alpha}) , 
	&&\strictpi \big( H_{l+1 , m+1} (\mathbfe{\alpha}) \big) := \strictpi (\mathbfe{\alpha}) .
\end{align*}
Note that $r \big( H_{l+1 , m+1} (\mathbfe{\alpha}) \big)$ is not necessarily equal to the rank of $H_{l+1 , m+1} (\mathbfe{\alpha})$. \\

We can now define the following sets in $\mathbb{F}_q^{n+1}$.
\begin{definition} \label{definition, notation for sets in terms of n,h,rho,pi,r}
Let $n,h \geq 0$. We define 
\begin{align*}
\mathcal{L}_n^h := &\{ \mathbfe{\alpha} \in \mathbb{F}_q^{n+1} : \text{ $\alpha_i = 0$ for $0 \leq i \leq h-1$} \} , \\
\mathcal{L}_n (r_1) := &\{ \mathbfe{\alpha} \in \mathbb{F}_q^{n+1} : r (\mathbfe{\alpha}) = r_1 \} , \\
\mathcal{L}_n^h (r_1) := &\{ \mathbfe{\alpha} \in \mathbb{F}_q^{n+1} : r (\mathbfe{\alpha}) = r_1 , \text{ $\alpha_i = 0$ for $0 \leq i \leq h-1$} \} , \\
\vspace{0.5em} \\
\mathcal{L}_n (r_1 , \rho_1 , \pi_1) := &\{ \mathbfe{\alpha} \in \mathbb{F}_q^{n+1} : r (\mathbfe{\alpha}) = r_1 , \rho (\mathbfe{\alpha}) = \rho_1 , \pi (\mathbfe{\alpha}) = \pi_1 \} , \\
\mathcal{L}_n^h (r_1 , \rho_1 , \pi_1) := &\{ \mathbfe{\alpha} \in \mathbb{F}_q^{n+1} : r (\mathbfe{\alpha}) = r_1 , \rho (\mathbfe{\alpha}) = \rho_1 , \pi (\mathbfe{\alpha}) = \pi_1 , \text{ $\alpha_i = 0$ for $0 \leq i \leq h-1$} \} , \\
\vspace{0.5em} \\
\strictmathcal{L}_n (r_1 , \rho_1 , \pi_1) := &\{ \mathbfe{\alpha} \in \mathbb{F}_q^{n+1} : r (\mathbfe{\alpha}) = r_1 , \strictrho (\mathbfe{\alpha}) = \rho_1 , \strictpi (\mathbfe{\alpha}) = \pi_1 \} , \\
\strictmathcal{L}_n^h (r_1 , \rho_1 , \pi_1) := &\{ \mathbfe{\alpha} \in \mathbb{F}_q^{n+1} : r (\mathbfe{\alpha}) = r_1 , \strictrho (\mathbfe{\alpha}) = \rho_1 , \strictpi (\mathbfe{\alpha}) = \pi_1 , \text{ $\alpha_i = 0$ for $0 \leq i \leq h-1$} \} .
\end{align*}
\end{definition}

It is helpful to keep in mind that, by definition, we always have $r_1 = \rho_1 + \pi_1$ above. Note also that for $n$ even, the set $\strictmathcal{L}_n (n_1 , n_1 , 0)$ is empty, because we always have $\strictrho (\mathbfe{\alpha}) \leq n_2 -1 = n_1 -1$. \\

Now, let us make some remarks that we will refer to later, but first we will need to define lower skew-triangular Hankel matrices.

\begin{definition}[Lower Skew-triangular Hankel Matrices]
Lower skew-triangular Hankel matrices are defined to be exactly the matrices $H_{n_1 , n_2} (\mathbfe{\alpha})$ for which $\mathbfe{\alpha} \in \mathcal{L}_n^{n_1}$ (for any $n \geq 0$). In particular, the first $n_1$ entries of $\mathbfe{\alpha}$ are zero, and so $H_{n_1 , n_2} (\mathbfe{\alpha})$ is of the form 
\begin{align*}
\begin{pNiceMatrix}
0 & \Cdots &  &  &  &  & 0 \\
\Vdots &  &  &  &  &  & \Vdots \\
0 & \Cdots &  &  &  &  & 0 \\
0 & \Cdots &  &  &  & 0 & \alpha_{n_1 +1} \\
0 & \Cdots &  &  & 0 & \alpha_{n_1 +1} & \alpha_{n_1 +2} \\
\Vdots &  &  & \Iddots & \Iddots & \Iddots & \Vdots \\
0 & \Cdots & 0 & \alpha_{n_1 +1} & \alpha_{n_1 +2} & \Cdots & \alpha_n \\
\end{pNiceMatrix} .
\end{align*}
Recall that if $\mathbfe{\alpha} \in \mathcal{L}_n^{n_1}$ then \emph{at least} the first $n_1$ entries of $\mathbfe{\alpha}$ that are zero, and there may be more. Thus $\alpha_{n_1 +1}$ is not necessarily non-zero above. It is also worth noting that if $n$ is even, then there must be at least one row of zeros at the top and at least one column of zeros on the left; while if $n$ is odd then there is at least one column of zeros on the left, but it is possible that there are no rows of zeros at the top.
\end{definition}

\begin{remark} \label{remark, rho_1 values dependent on h}
Suppose that $\mathbfe{\alpha} \in \mathcal{L}_n^h$. Then, we must have that $\rho (\mathbfe{\alpha}) = 0$ or $\rho (\mathbfe{\alpha}) \geq h+1$, and $\strictrho (\mathbfe{\alpha}) = 0$ or $\strictrho (\mathbfe{\alpha}) \geq h+1$. This is not surprising. Indeed, if the first $h$ entries of $\mathbfe{\alpha}$ are zero, then $H_{l,l} (\mathbfe{\alpha})$ is strictly lower skew-triangular for $l \leq h$, and thus it is not invertible. \\

Note this implies that if $h \geq n_1$, then $\rho (\mathbfe{\alpha}) = 0$ (and $\strictrho (\mathbfe{\alpha}) = 0$). Furthermore, $H_{n_1 , n_2} (\mathbfe{\alpha})$ is lower skew-triangular. In particular, if we let $z \geq h$ be such that there are exactly $z$ zeros at the start of $\mathbfe{\alpha}$ (Recall that if $\mathbfe{\alpha} \in \mathcal{L}_n^h$ then $\mathbfe{\alpha}$ has \emph{at least} $h$ zeros at the start, but it could have more), then $H_{n_1 , n_2} (\mathbfe{\alpha})$ is of the form
\begin{align*}
\begin{pNiceMatrix}
0 & \Cdots &  &  &  &  & 0 \\
\Vdots &  &  &  &  &  & \Vdots \\
0 & \Cdots &  &  &  &  & 0 \\
0 & \Cdots &  &  &  & 0 & \alpha_{z+1} \\
0 & \Cdots &  &  & 0 & \alpha_{z+1} & \alpha_{z+2} \\
\Vdots &  &  & \Iddots & \Iddots & \Iddots & \Vdots \\
0 & \Cdots & 0 & \alpha_{z+1} & \alpha_{z+2} & \Cdots & \alpha_n \\
\end{pNiceMatrix} 
\end{align*}
with $\alpha_{z+1} \neq 0$. In particular, we see that the rank of $H_{n_1 , n_2} (\mathbfe{\alpha})$ is $n-z$. \\

We also note a converse result. If $\mathbfe{\alpha} \in \mathcal{L}_n$ and $\rho (\mathbfe{\alpha}) = 0$ (or $\strictrho (\mathbfe{\alpha}) = 0$), then $H_{n_1 , n_2} (\mathbfe{\alpha})$ is lower skew-triangular. Indeed, by definition of $\rho$ (and $\strictrho$), we have that $H_{1,1} (\mathbfe{\alpha}) , \ldots , H_{n_1,n_1} (\mathbfe{\alpha})$ are all not invertible, and then an inductive argument tells us that the first $n_1$ entries of $\mathbfe{\alpha}$ are zero.
\end{remark}

Theorem 2.3.1 of \cite{Yiasemides2021_VariCorrDivFuncFpTHankelMatr_ArXiv_v2} gives us the number of elements in the sets we have defined above:

\begin{lemma} \label{lemma, number of sequences of given rhopi form}
Let $n \geq 0$ and $0 \leq h \leq n+1$, and consider $\mathscr{L}_{n}^{h} (r , \rho_1 , \pi_1 )$. Let us also define $\EvenInd (n)$ to be $1$ if $n$ is even, and $0$ if $n$ is odd. 

\begin{itemize}
\item Suppose $0 \leq r \leq \min \{ n_1 - \EvenInd (n) , n-h+1 \}$. Then,
\begin{align*}
\lvert \mathscr{L}_{n}^{h} (r , 0 , r ) \rvert
= \begin{cases}
1 &\text{ if $r=0$,} \\
(q-1) q^{r-1} &\text{ if $r > 0$.} 
\end{cases}  \\
\end{align*}

\item Suppose that $h+1 \leq \rho_1 \leq n_1 - 1$ and $0 \leq \pi_1 \leq n_1 - \rho_1 - \EvenInd (n)$. Then, 
\begin{align*}
\lvert \mathscr{L}_{n}^{h} (\rho_1 + \pi_1 , \rho_1 , \pi_1 ) \rvert
= \begin{cases}
(q-1) q^{2 \rho_1 - h - 1} &\text{ if $\pi_1 = 0$,} \\
(q-1)^2 q^{2 \rho_1 + \pi_1 - h -2}  &\text{ if $\pi_1 > 0$.} 
\end{cases} \\
\end{align*}

\item Suppose $h+1 \leq n_1$. We have
\begin{align*}
\lvert \mathscr{L}_{n}^{h} (n_1 , n_1 , 0 ) \rvert
= (q-1) q^{n-h} . \\
\end{align*}

\item Consider $\mathscr{L}_n^h (r )$. We have
\begin{align*}
\lvert \mathscr{L}_n^h (r ) \rvert
= \begin{cases}
1 &\text{ if $r=0$,} \\
(q-1) q^{r-1} &\text{ if $1 \leq r \leq \min \{ h , n-h+1 \} $,} \\
(q^2 -1) q^{2r-h-2} &\text{ if $h+1 \leq r \leq n_1 - 1$,} \\ 
q^{n-h+1} - q^{2n_1 - h -2} &\text{ if $r = n_1$ (which is only possible if $h+1 \leq n_1$).}
\end{cases}
\end{align*}
\end{itemize}
\end{lemma}

Now that we have defined the \rhopi-characteristic of a Hankel matrix, we can introduce the \rhopi-form of such matrices. This involves the application of certain row operations, and the resulting matrix is useful in understanding the kernel structure. We will only give a brief summary of this so that it is clear to the reader what the importance of the \rhopi-characteristic is in relation to the kernel, before explicitly giving results on the kernel structure. For more details, we refer the reader to Section 2 of \cite{Yiasemides2021_VariCorrDivFuncFpTHankelMatr_ArXiv_v2}. \\

Suppose we have $\mathbfe{\alpha} \in \mathcal{L}_n (r_1 , \rho_1 , \pi_1)$ with $1 \leq \rho_1 \leq n_1 -1$, and consider the matrix $H_{n_1 , n_2} (\mathbfe{\alpha})$. By definition of $\rho (\mathbfe{\alpha})$ we can see that the top-left submatrix $H_{\rho_1 , \rho_1} (\mathbfe{\alpha})$ is invertible. In particular, there is a solution $\mathbf{x} = (x_0 , \ldots , x_{\rho_1 -1}) \in \mathbb{F}_q^{\rho_1}$ to
\begin{align} \label{statement, H_(rho_1 , rho_1) (alpha) x = right vector}
H_{\rho_1 , \rho_1} (\mathbfe{\alpha}) \mathbf{x}
= \begin{pmatrix} \alpha_{\rho_1} \\ \alpha_{\rho_1 +1} \\ \vdots \\ \alpha_{2\rho_1 -1} \end{pmatrix} .
\end{align}
The vector on the right side is simply the column vector directly to the right of the submatrix $H_{\rho_1 , \rho_1} (\mathbfe{\alpha})$ in $H_{n_1 , n_2} (\mathbfe{\alpha})$. By symmetry, it is also the transpose of the row directly below $H_{\rho_1 , \rho_1} (\mathbfe{\alpha})$. Now let $R_i$ be the $i$-th row of $H_{n_1 , n_2} (\mathbfe{\alpha})$ and apply the row operations
\begin{align*} 
R_i
\longrightarrow R_i - (x_0 , \ldots , x_{\rho_1 -1}) 
	\begin{pmatrix} R_{i-\rho_1} \\ \vdots \\ R_{i-1} \end{pmatrix}
\end{align*}
for $i = n_1 , n_1 -1 , \ldots , \rho_1 +1$ in that order. The resulting matrix is of the form

\begin{align} \label{statement, rhopi-form of H_(n_1 , n_2) (alpha)}
\left(
\begin{array}{c|c}
H_{\rho_1 , \rho_1} (\mathbfe{\alpha}) & \mathbfe{*} \\
\hline
\mathbf{0} & 
\begin{NiceMatrix}
0 & \Cdots &  &  &  &  & 0 \\
\Vdots &  &  &  &  &  & \Vdots \\
0 & \Cdots &  &  &  &  & 0 \\
0 & \Cdots &  &  &  & 0 & \lambda \\
0 & \Cdots &  &  & 0 & \lambda & \beta_2 \\
\Vdots &  &  & \Iddots & \Iddots & \Iddots & \Vdots \\
0 & \Cdots & 0 & \lambda & \beta_2 & \Cdots & \beta_{\pi_1} \\
\end{NiceMatrix}
\end{array}
\right) .
\end{align}
The top-right quadrant $\mathbfe{*}$ is the same as the corresponding submatrix in $H_{n_1 , n_2} (\mathbfe{\alpha})$, but we do not need to express it explicitly; $\lambda$ is some element in $\mathbb{F}_q^*$; and $\beta_2 , \ldots , \beta_{\pi_1} \in \mathbb{F}_q$. The full and rigorous explanation of this is given in Section 2.2 of \cite{Yiasemides2021_VariCorrDivFuncFpTHankelMatr_ArXiv_v2}, but we can give some intuitive indications as to its validity. First, note that the entries directly below $H_{\rho_1 , \rho_1} (\mathbfe{\alpha})$ must be zero by definition of $\mathbf{x}$ and the row operations we applied. Second, the row operations we applied certainly changed the bottom two quadrants, but they preserved the property that the matrix formed by these bottom two quadrants is Hankel. Thus, all that remains to be shown is that the first non-zero skew-diagonal of the bottom-right quadrant is the $\pi_1$-th skew diagonal from the end. This can be proved inductively by considering the submatrices $H_{l,l} (\mathbfe{\alpha})$ for $l = \rho_1 +1 , \ldots n_1$ and using the fact that all of these are not invertible (by definition of $\rho (\mathbfe{\alpha})$). This will justify the skew-diagonals that are zero. The fact that $\lambda \neq 0$ can be seen by the fact that we require the rank of $H_{n_1 , n_2} (\mathbfe{\alpha})$ to be $r_1 = \rho_1 + \pi_1$. \\

The matrix (\ref{statement, rhopi-form of H_(n_1 , n_2) (alpha)}) is the \rhopi-form of $H_{n_1 , n_2} (\mathbfe{\alpha})$. Above, we only considered the case where $1 \leq \rho_1 \leq n_1 -1$, but (\ref{statement, rhopi-form of H_(n_1 , n_2) (alpha)}) can be used to extend the definition naturally to the cases $\rho_1 = 0$ and $\rho_1 = n_1$. \\

When $\rho_1 = 0$, $H_{0,0} (\mathbfe{\alpha})$ is interpreted as being an empty matrix, as are the top-right and bottom-left quadrants of (\ref{statement, rhopi-form of H_(n_1 , n_2) (alpha)}) since they have zero rows and columns, respectively. In particular, we are left with the lower skew-triangular Hankel matrix from the bottom-right. One may ask what row operations we apply to obtain this from the original matrix. The answer is none (which can be interpreted as applying the trivial row-operations). Indeed, Remark \ref{remark, rho_1 values dependent on h} tells us that when $\rho_1 = 0$, the matrix $H_{n_1 , n_2} (\mathbfe{\alpha})$ is lower skew-triangular and of the same form as the bottom-right quadrant in (\ref{statement, rhopi-form of H_(n_1 , n_2) (alpha)}). Therefore, the \rhopi-form of $H_{n_1 , n_2} (\mathbfe{\alpha})$ is just itself. \\

When $\rho_1 = n_1$, the top-left quadrant of (\ref{statement, rhopi-form of H_(n_1 , n_2) (alpha)}) has as many rows as the entire matrix, and so the bottom two quadrants have zero rows and are therefore empty matrices. Hence, (\ref{statement, rhopi-form of H_(n_1 , n_2) (alpha)}) is just the original matrix. That is, the \rhopi-form of $H_{n_1 , n_2} (\mathbfe{\alpha})$ is just itself. \\

So far, we have defined the \rhopi-form only for $H_{n_1 , n_2} (\mathbfe{\alpha})$. A similar definition holds for $H_{l+1 , m+1} (\mathbfe{\alpha})$ when $l+m=n$, but we are particularly interested in the case where both $l+1 , m+1 \geq r_1$ hold. In this case, again we can apply the row operations
\begin{align*} 
R_i
\longrightarrow R_i - (x_0 , \ldots , x_{\rho_1 -1}) 
	\begin{pmatrix} R_{i-\rho_1} \\ \vdots \\ R_{i-1} \end{pmatrix} ,
\end{align*}
but for $i = l+1 , l , \ldots , \rho_1 +1$ this time. We will obtain a matrix of the form (\ref{statement, rhopi-form of H_(n_1 , n_2) (alpha)}) again. Of course, the number of rows and columns are different, but the most significant difference this causes is the number of zero-columns/zero-rows that appear on the left and top of the bottom-right quadrant; and in the boundary cases $m+1 =r_1$ or $l+1 =r_1$ we will have no zero-columns or zero-rows, respectively. \\

Finally, we can define the strict \rhopi-form, which, as the name suggests, is based on the strict \rhopi-characteristic. It is essentially the same approach used to obtain (\ref{statement, rhopi-form of H_(n_1 , n_2) (alpha)}), but we take $\rho_1 = \strictrho (\mathbfe{\alpha})$ and $\pi_1 = \strictpi (\mathbfe{\alpha})$, instead of $\rho_1 = \rho (\mathbfe{\alpha})$ and $\pi_1 = \pi (\mathbfe{\alpha})$. The only time where the strict \rhopi-form differs from the standard \rhopi-form is when $n$ is even and $\mathbfe{\alpha} \in \mathcal{L}_n (n_1)$. In this case, $H_{n_1 , n_2} (\mathbfe{\alpha})$ has full rank and so the standard \rhopi-form is simply the matrix itself. Whereas the strict \rhopi-form is of the form
\begin{align*}
\left(
\begin{array}{c|c}
H_{\rho_1 , \rho_1} (\mathbfe{\alpha}) & \mathbfe{*} \\
\hline
\mathbf{0} & 
\begin{NiceMatrix}
0 & \Cdots &  & 0 & \lambda \\
\Vdots &  & \Iddots & \lambda & \beta_2 \\
 & \Iddots & \Iddots & \Iddots & \Vdots \\
0 & \lambda & \Iddots &  &  \\
\lambda & \beta_2 & \Cdots &  & \beta_{\pi_1} \\
\end{NiceMatrix}
\end{array}
\right) .
\end{align*}
Let us now make two remarks that we will refer to later.

\begin{remark} \label{remark, rank of Hankel matrix is minimum of rows, columns, and r}
Suppose $\mathbfe{\alpha} \in \mathcal{L}_n (r_1 , \rho_1 , \pi_1)$, for some $r_1 , \pi_1 , \rho_1$. We have that
\begin{align*}
\rank H_{l+1 , m+1} (\mathbfe{\alpha})
= \min \{ r_1 , l+1 , m+1 \} .
\end{align*}
This is not difficult to see. Indeed, if $l+1 , m+1 \geq r_1$, then $H_{l+1 , m+1} (\mathbfe{\alpha})$ has \rhopi-form of the form (\ref{statement, rhopi-form of H_(n_1 , n_2) (alpha)}). This clearly has rank $\rho_1 + \pi_1 = r_1$, and since the row operations we applied to obtain the \rhopi-form do not alter the rank, we have that $\rank H_{l+1 , m+1} (\mathbfe{\alpha}) = r_1$. \\

On the other hand, if $l+1 < r_1$, then we note that the rows of $H_{l+1 , m+1} (\mathbfe{\alpha})$ can be truncated to become the first $l+1$ rows of $H_{\rho_1 , \rho_1} (\mathbfe{\alpha})$. Since $H_{\rho_1 , \rho_1} (\mathbfe{\alpha})$ has full rank, we must have that the rows of $H_{l+1 , m+1} (\mathbfe{\alpha})$ are linearly independent, and thus the rank of $H_{l+1 , m+1} (\mathbfe{\alpha})$ is $l+1$. If $m+1 < r_1$, then a similar argument applies but with the columns instead of the rows. 
\end{remark}

\begin{remark} \label{lemma, effect of removing last entry of alpha}
Suppose $n$ is even and $\mathbfe{\alpha} \in \strictmathcal{L}_n (r_1 , \rho_1 , \pi_1)$, for some $r_1 , \pi_1 , \rho_1$. Let $\mathbfe{\alpha}'$ be the sequence obtained by removing the last entry of $\mathbfe{\alpha}$. It is not difficult to see that $H_{n_1 -1 , n_1} (\mathbfe{\alpha}')$ is the matrix obtained by removing the last row of $H_{n_1 , n_1} (\mathbfe{\alpha})$. Furthermore, the \rhopi-form of $H_{n_1 -1 , n_1} (\mathbfe{\alpha}')$ is the matrix obtained by removing the last row of the \rhopi-form of $H_{n_1 , n_1} (\mathbfe{\alpha})$. In particular, we see that if $\strictpi (\mathbfe{\alpha}) \geq 1$, then $\mathbfe{\alpha}' \in \strictmathcal{L}_{n-1} (r_1 -1 , \rho_1 , \pi_1 -1)$; while if $\strictpi (\mathbfe{\alpha}) = 0$, then $\mathbfe{\alpha}' \in \strictmathcal{L}_{n-1} (r_1 , \rho_1 , \pi_1)$. Furthermore, this implies
\begin{align*}
\lvert \kernel H_{n_1 , n_1} (\mathbfe{\alpha}) \rvert
= \begin{cases}
q^{-1} \lvert \kernel H_{n_1 -1 , n_1} (\mathbfe{\alpha}') \rvert &\text{ if $\strictpi (\mathbfe{\alpha}) \geq 1$,} \\
\lvert \kernel H_{n_1 -1 , n_1} (\mathbfe{\alpha}') \rvert &\text{ if $\strictpi (\mathbfe{\alpha}) = 0$.}
\end{cases}
\end{align*}
\end{remark}

Continuing our discussion, using these \rhopi-forms, and the fact that the row operations do not affect the kernel, it is possible to give an intuitive explanation of the kernel structure of Hankel matrices. First consider the matrix $H_{n+2-r_1 , r_1} (\mathbfe{\alpha})$. The \rhopi-form is of the form
\begin{align*}
\left(
\begin{array}{c|c}
H_{\rho_1 , \rho_1} (\mathbfe{\alpha}) & \mathbfe{*} \\
\hline
\mathbf{0} & 
\begin{NiceMatrix}
0 & \Cdots &  &  & 0 \\
\Vdots &  &  &  & \Vdots \\
 &  &  &  & 0 \\
 &  &  & 0 & \lambda \\
 &  & 0 & \lambda & \beta_2 \\
 & \Iddots & \Iddots & \Iddots & \Vdots \\
0 & \lambda & \Cdots &  &  \\
\lambda & \beta_2 & \Cdots &  & \beta_{\pi_1} \\
\end{NiceMatrix}
\end{array}
\right) .
\end{align*}
Clearly, the kernel here consists only of the zero vector, as both the top-left and bottom-right quadrants have full column rank. In fact, it is not difficult to see from this that $H_{l+1 , m+1} (\mathbfe{\alpha})$ has full column rank, and thus a trivial kernel, for all $m+1 \leq \rho_1 + \pi_1 =r_1$. Now consider $H_{n+1-r_1 , r_1 +1} (\mathbfe{\alpha})$, which has one more column and one less row compared to $H_{n+2-r_1 , r_1} (\mathbfe{\alpha})$ (but, of course, they have the same underlying sequence $\mathbfe{\alpha}$). The matrix $H_{n+1-r_1 , r_1 +1} (\mathbfe{\alpha})$ has \rhopi-form of the form
\begin{align*}
\left(
\begin{array}{c|c}
H_{\rho_1 , \rho_1} (\mathbfe{\alpha}) & \mathbfe{*} \\
\hline
\mathbf{0} & 
\begin{NiceMatrix}
0 & \Cdots &  &  & 0 \\
\Vdots &  &  &  & \Vdots \\
0 & \Cdots &  &  & 0 \\
0 & \Cdots &  & 0 & \lambda \\
0 & \Cdots & 0 & \lambda & \beta_2 \\
\Vdots & \Iddots & \Iddots & \Iddots & \Vdots \\
0  & \lambda & \beta_2 & \Cdots & \beta_{\pi_1} \\
\end{NiceMatrix}
\end{array}
\right) .
\end{align*}
Clearly, this has a one-dimensional kernel, and we can deduce that the kernel consists of scalar multiples of $[A_1]_{r_1}$, where 
\begin{align*}
A_1 := -x_0 - x_1 T - \ldots - x_{\rho_1 -1} T^{\rho_1 -1} + T^{\rho_1}
\end{align*}
and $x_0 , \ldots , x_{\rho_1 -1}$ is as in (\ref{statement, H_(rho_1 , rho_1) (alpha) x = right vector}). We now make the following two observations:

\begin{enumerate}
\item Consider $H_{l+1 , m+1} (\mathbfe{\alpha})$ where both $l+1 , m+1 \geq r_1 +1$ hold. As we have already stated, the \rhopi-form is of the form (\ref{statement, rhopi-form of H_(n_1 , n_2) (alpha)}). From this, we can see that the rank of $H_{l+1 , m+1} (\mathbfe{\alpha})$ is always $r_1 = \rho_1 + \pi_1$, regardless of the values of $l+1,m+1$ in the given ranges. However, we note that $H_{l , m+2} (\mathbfe{\alpha})$ has one more row than $H_{l+1 , m+1} (\mathbfe{\alpha})$, and so the dimension of the kernel increases by one every time we remove one row and add one column.

\item Suppose $\mathbf{f} \in \mathbb{F}_q^{m+1}$ is in the kernel of $H_{l+1 , m+1} (\mathbfe{\alpha})$. It is not difficult to see that the vectors
\begin{align*}
\begin{pmatrix} \mathbf{f} \\ \hline 0 \end{pmatrix}
\hspace{3em} \text{ and } \hspace{3em}
\begin{pmatrix} 0 \\ \hline \mathbf{f} \end{pmatrix}
\end{align*}
are in the kernel of $H_{l , m+2} (\mathbfe{\alpha})$, as is any linear combination (over $\mathbb{F}_q$). If we associate $\mathbf{f}$ with a polynomial $F$, then the above says that if $[F]_m$ is in the kernel of $H_{l+1 , m+1} (\mathbfe{\alpha})$ then $[BF]_{m+1}$ is in the kernel of $H_{l , m+2} (\mathbfe{\alpha})$ for any $B \in \mathcal{A}_{\leq 1}$; and more generally, $[BF]_{m+k}$ is in the kernel of $H_{l-k+1 , m+k+1} (\mathbfe{\alpha})$ for any $B \in \mathcal{A}_{\leq k}$. \\
\end{enumerate}

Based on these two points, we can deduce that for $l+1 \geq r_1$, the kernel of $H_{l+1 , m+1} (\mathbfe{\alpha})$ consists exactly of the vectors $[B_1 A_1]_m$ with $B_1 \in \mathcal{A}_{m-r_1}$. This is non-trivial when $m+1 \geq r_1 +1$ (recall that for all positive integers $i$, the set $\mathcal{A}_{-i}$ is defined to be $\{ 0 \}$). \\

Now consider the matrices $H_{l+1 , m+1} (\mathbfe{\alpha})$ and $H_{l , m+2} (\mathbfe{\alpha})$ when $l+1 \leq r$. By similar reasoning as above, we can see that as we go from the former to the latter, the rank decreases by one and the number of columns increases by one. Thus, the dimension of the kernel increases by two. In particular, a new polynomial $A_2$, independent of $A_1$ (see below for a formal exposition), must appear in the kernel of $H_{r-1 , n+3-r} (\mathbfe{\alpha})$, and its multiples appear in kernels of $H_{l+1 , m+1} (\mathbfe{\alpha})$ for $l+1 \leq r-2$. Formally, we have the following result, which is Theorem 2.4.4 from \cite{Yiasemides2021_VariCorrDivFuncFpTHankelMatr_ArXiv_v2} (this result is originally given in \cite{HeinigRost1984_AlgMethToeplitzMatrOperat}).

\begin{lemma} \label{lemma, kernel of Hankel matrices}
Let $\mathbfe{\alpha} \in \mathcal{L}_n (r_1 , \rho_1 , \pi_1)$. Then, there exist coprime polynomials $A_1 (\mathbfe{\alpha}) \in \mathcal{M}_{\rho_1}$ and $A_2 (\mathbfe{\alpha}) \in \mathcal{A}_{\leq n-r_1 +2}$ such that for all $l,m \geq 0$ with $l+m=n$ we have
\begin{align*}
\kernel H_{l+1 , m+1} (\mathbfe{\alpha})
= \bigg\{ [B_1 A_1 (\mathbfe{\alpha}) + B_2 A_2 (\mathbfe{\alpha})]_m : \substack{B_1 , B_2 \in \mathcal{A} \\ \degree B_1 \leq m-r_1 \\ \degree B_2 \leq m - (n-r_1+2)} \bigg\} .
\end{align*}
\end{lemma}

\begin{definition}[Characteristic Polynomials]
The polynomials $A_1 (\mathbfe{\alpha})$ and $A_2 (\mathbfe{\alpha})$ are called the first and second characteristic polynomials of $\mathbfe{\alpha}$, respectively. We may also say they are the characteristic polynomials of $H_{l+1 , m+1} (\mathbfe{\alpha})$ for any $l,m \geq 0$ with $l+m=n$. This should not be confused with the characteristic polynomial associated to the determinant of a square matrix.
\end{definition}

Let us make a couple of remarks that we will require later.

\begin{remark} \label{remark, kernel of Hankel matrices}
The first characteristic polynomial $A_1 (\mathbfe{\alpha})$ is unique. Non-zero scalar multiples can also be used in its place, but we simply take the monic case. The second characteristic polynomial $A_2 (\mathbfe{\alpha})$ is unique only up to non-zero scalar multiplication and addition of polynomials of the form $B_1 A_1 (\mathbfe{\alpha})$ for $B_1 \in \mathcal{A}_{\leq n-2r_1 +2}$. In the special case where $\rho_1 = r_1$, and thus $\degree A_1 (\mathbfe{\alpha}) = r_1$, this means we can take $A_2 (\mathbfe{\alpha})$ to be the representative modulo $A_1 (\mathbfe{\alpha})$ of degree $< \degree A_1 (\mathbfe{\alpha}) = r_1$ that is monic. This is unique. In the case where $\rho_1 < r_1$, there is not a natural unique representation to take. \\

In the paragraph above, we explain to what extent the characteristic polynomials are unique for a given $\mathbfe{\alpha}$. We can also ask whether two characteristic polynomials $A_1 , A_2$ have a unique sequence $\mathbfe{\alpha}$ associated to them. The answer is that there are actually $q-1$ possible associated $\mathbfe{\alpha} \in \mathbb{F}_q^{n+1}$, because we can multiply any such $\mathbfe{\alpha}$ by an element in $\mathbb{F}_q^*$ and this does not affect kernel, and thus it does not affect the characteristic polynomials. Note if $\mathbfe{\alpha} = \mathbf{0}$, then multiplying by an element in $\mathbb{F}_q^*$ does nothing, and so in this case $\mathbf{0}$ does not share its characteristic polynomials. \\

Now, we considered the case $\mathbfe{\alpha} = \mathbf{0}$ above, but this raises the question if $A_2 (\mathbfe{\alpha})$ is well-defined. Indeed, $A_2 (\mathbfe{\alpha})$ appears in the kernel of at least one $H_{l+1 , m+1} (\mathbfe{\alpha})$ with $l+m=n$ only when $r_1 \geq 2$. However, it is useful to define it even when $r_1 \leq 1$ because it may be used when considering extensions of $\mathbfe{\alpha}$ (although we do not consider such extensions in this paper). In fact, the definition of $A_2 (\mathbfe{\alpha})$ given in the lemma (that is, the degree restrictions and coprimality condition) accommodates for this and it gives
\begin{align*}
A_2 (\mathbfe{\alpha})
= \begin{cases}
0 &\text{ if $\mathbfe{\alpha} \in \mathcal{L}_n (0,0,0)$ (that is, $\mathbfe{\alpha} = \mathbf{0}$),} \\
1 &\text{ if $\mathbfe{\alpha} \in \mathcal{L}_n (1,1,0)$,} \\
T^{n+1} &\text{ if $\mathbfe{\alpha} \in \mathcal{L}_n (1,0,1)$.} \\
\end{cases}
\end{align*}
(For the first two cases, this is also consistent with the unique representation of $A_2 (\mathbfe{\alpha})$ modulo $A_1 (\mathbfe{\alpha})$ mentioned above). \\

Finally, we note that $A_2 (\mathbfe{\alpha})$ is non-zero except for the single trivial case given above, and so we usually take it to be monic (even though there is no natural unique representation to take when $\rho_1 < r_1$).
\end{remark}

\begin{remark} \label{remark, pi_1 = 0 iff vector in kernel with no zero at end}
Let $\mathbfe{\alpha} \in \mathcal{L}_n (r_1 , \rho_1 , \pi_1)$, and suppose $l+m=n$ with $l+1 \geq r_1$. Lemma \ref{lemma, kernel of Hankel matrices} tells us that 
\begin{align*}
\kernel H_{l+1 , m+1} (\mathbfe{\alpha})
= \bigg\{ [B_1 A_1 (\mathbfe{\alpha})]_m : \substack{B_1 \in \mathcal{A} \\ \degree B_1 \leq m-r_1 } \bigg\} .
\end{align*}
We can see that the kernel has a vector with non-zero final entry if and only if  $\degree B_1 A_1 (\mathbfe{\alpha}) = m$. In turn, this occurs if and only if $(r_1 , \rho_1 , \pi_1) = (r_1 , r_1 , 0)$. \\

This provides a useful equivalence for the condition $\pi (\mathbfe{\alpha}) = 0$, which we will require later. Note also that if $\pi_1 = 0$, then the number of elements in $\kernel H_{l+1 , m+1} (\mathbfe{\alpha})$ is just $q$ times the number of vectors in the kernel with final entry equal to zero. \\

We must keep in mind that the above only applies when $l+1 \geq r_1$, because when $l+1 < r_1$ the second characteristic polynomial appears in the kernel and interferes with our reasoning above.
\end{remark}

So, we now have an understanding of the kernel structure of Hankel matrices, and how the \rhopi-form is used in determining this. We now wish to state some results on the value distribution of quadratic forms associated to square Hankel matrices. We undertook this in \cite{Yiasemides2022_VariSumTwoSquareOverIntervalFqT_I_Arxiv}. Let us give a brief indication of our approach to this. \\

The strict \rhopi-characteristic allows us to obtain a (strict) \rhopi-form of the form (\ref{statement, rhopi-form of H_(n_1 , n_2) (alpha)}) even when our matrix has full rank. In particular, for square Hankel matrices, we can apply this inductively, first to $H_{n_1 , n_1} (\mathbfe{\alpha})$, second to the submatrix $H_{\rho_1 , \rho_1} (\mathbfe{\alpha})$, and so on. Ultimately, this is just an application of row operations. We then also apply the same operations but for the columns. In the context of quadratic forms, we have simply undertaken a change of basis, which does not change the value distribution. The benefit is that it is easier to determine the value distribution of the matrix obtained after the operations are applied. Indeed, ultimately we obtain a matrix that is block-diagonal and each block is a square, lower skew-triangular Hankel matrix (that is, the same form as the bottom-right quadrant of (\ref{statement, rhopi-form of H_(n_1 , n_2) (alpha)})). This is called the reduced \rhopi-form, and in the context of quadratic forms we can study each block individually, and their lower skew-triangular Hankel form is particularly easy to work with. The details of this can be found in Section 3 of \cite{Yiasemides2022_VariSumTwoSquareOverIntervalFqT_I_Arxiv}, but for this paper we only need the following result which follows from that section (particularly Lemma 3.3.4):

\begin{lemma} \label{lemma, quadratic form values over monics and nonmonics}
Let $l \geq 1$ and $\mathbfe{\alpha} \in \strictmathcal{L}_{2l} (r_1 , \rho_1 , \pi_1)$ ($r_1$ can take any value in $[0,1, \ldots , l+1]$). Then,
\begin{align*}
\bigg\lvert \sum_{E \in \mathcal{A}_{\leq l}}
	\psi \big( [E]_l^T H_{l+1 , l+1} (\mathbfe{\alpha}) [E]_l \big) \bigg\rvert^2
= q^{2l+2-r_1}
= q^{l+1} \lvert \kernel H_{l+1 , l+1} (\mathbfe{\alpha}) \rvert
\end{align*}
and
\begin{align*}
\bigg\lvert \sum_{E \in \mathcal{M}_{l}}
	\psi \big( [E]_l^T H_{l+1 , l+1} (\mathbfe{\alpha}) [E]_l \big) \bigg\rvert^2
=\begin{cases}
q^{2l-r_1} = q^{l-1} \lvert \kernel H_{l+1 , l+1} (\mathbfe{\alpha}) \rvert &\text{ if $\pi_1 = 0$,} \\
q^{2l+1-r_1} = q^{l} \lvert \kernel H_{l+1 , l+1} (\mathbfe{\alpha}) \rvert &\text{ if $\pi_1 = 1$,} \\
0 &\text{ if $\pi_1 \geq 2$.} 
\end{cases}
\end{align*}
\end{lemma}

We have now covered the results from \cite{Yiasemides2021_VariCorrDivFuncFpTHankelMatr_ArXiv_v2, Yiasemides2022_VariSumTwoSquareOverIntervalFqT_I_Arxiv} that we require. For the new results that we will prove, we will require the following definition.

\begin{definition}
Suppose we have a vector $\mathbf{v} = (v_1 , v_2 , \ldots , v_l) \in \mathbb{F}_q^{l}$ and a sequence $\mathbf{s} = (s_1 , s_2 , \ldots )$ of length $m \geq l$, where $m$ may be finite or infinity. We define $\mathbf{s} \odot \mathbf{v}$ to be the sequence $\mathbf{t} = (t_1 , t_2 , \ldots )$ of length $m-l+1$ (we define $\infty -k$ to be $\infty$ for all integers $k$) such that
\begin{align*}
t_i
= \begin{pmatrix} s_i & s_{i+1} & \cdots & s_{i+l-1} \end{pmatrix}
	\begin{pmatrix}  v_1 \\ v_2 \\ \vdots \\ v_l \end{pmatrix}
\end{align*}
for all $i$ that index $\mathbf{t}$. That is, we are taking the dot product between $\mathbf{v}$ and successive blocks of $\mathbf{s}$ of same length as $\mathbf{v}$, and forming a new sequence.
\end{definition}

We observe that for $\mathbfe{\alpha} \in \mathcal{L}_n (r_1)$ the vector $\mathbfe{\alpha} \odot [A_1 (\mathbfe{\alpha})]_{r_1}$ is simply the zero vector in $\mathbb{F}_q^{n-r_1 +1}$. Recall that if $\rho_1 = r_1$ then $[A_1 (\mathbfe{\alpha})]_{r_1}$ is a vector with non-zero final entry. In particular, this means that if $\mathbfe{\alpha} \in \mathcal{L}_n (r_1 , r_1 , 0)$ we can define an infinite sequence $\overrightarrow{\mathbfe{\alpha}}$ such that the first $n+1$ terms are exactly those in $\mathbfe{\alpha}$ and the later terms are defined by the condition $\overrightarrow{\mathbfe{\alpha}} \odot [A_1 (\mathbfe{\alpha})]_{r_1} = \mathbf{0}$. Clearly, $\overrightarrow{\mathbfe{\alpha}}$ is unique for every $\mathbfe{\alpha}$. Also, $A_1 (\mathbfe{\alpha})$ defines a recurrence relation on $\overrightarrow{\mathbfe{\alpha}}$ (and hence $\mathbfe{\alpha}$ as well) and we will say it has length $r_1 +1$ (this is not necessarily standard terminology). Note that no shorter recurrence relation can exist, otherwise this would ultimately imply that $H_{n+2-r_1 , r_1} (\mathbfe{\alpha})$ has a non-trivial vector in its kernel, which contradicts the fact that $r (\mathbfe{\alpha}) = r_1$. 

\begin{lemma} \label{lemma, bijection from quasiregular Hankel matrices to M_r times A_<r}
Consider $\mathcal{L}_n^h (r,r,0)$ with $2 < r \leq n_2-1$ and $h<r$. The following map is bijective:
\begin{align*}
\mathcal{L}_n^h (r,r,0) \rightarrow &\big\{ (A,B) \in \mathcal{M}_r \times \mathcal{A}_{<r-h} : (A,B) \text{ coprime} \big\} \\
\mathbfe{\alpha} \mapsto &\Big( A_1 (\mathbfe{\alpha}) , A_1 (\mathbfe{\alpha}) \times L (\overrightarrow{\mathbfe{\alpha}}) \Big) ;
\end{align*}
where, if $\overrightarrow{\mathbfe{\alpha}} = (\alpha_0 , \alpha_1 , \ldots )$, then we define the Laurent series
\begin{align*}
L (\overrightarrow{\mathbfe{\alpha}})
:= \sum_{i = 0}^{\infty} \alpha_{i} T^{-i}  ;
\end{align*}
and $A_1 (\mathbfe{\alpha}) \times L (\overrightarrow{\mathbfe{\alpha}})$ is the standard multiplication of Laurent series.
\end{lemma}

\begin{proof}
First we note that $A_1 (\mathbfe{\alpha}) \times L (\overrightarrow{\mathbfe{\alpha}})$ is a polynomial of degree $<r$, which we will denote by $B$. This follows from the fact that $\overrightarrow{\mathbfe{\alpha}} \odot [A_1 (\mathbfe{\alpha})]_{r} = \mathbf{0}$, which we established above. Furthermore, since the first $h$ entries of $\overrightarrow{\mathbfe{\alpha}}$ are zero, we have that $L (\overrightarrow{\mathbfe{\alpha}}) = \sum_{i = h}^{\infty} \alpha_{i} T^{-i}$, and so $B$ must actually have degree $<r-h$. \\

To prove that $A_1 (\mathbfe{\alpha})$ and $B$ are coprime, we note that $L (\overrightarrow{\mathbfe{\alpha}})$ is the Laurent series for the rational function $\frac{B}{A_1 (\mathbfe{\alpha})}$. In particular, if $A_1 (\mathbfe{\alpha})$ and $B$ were not coprime, then there exist some $C,D$ with $\degree C < \degree A_1 (\mathbfe{\alpha})$ such that
\begin{align*}
\frac{B}{A_1 (\mathbfe{\alpha})}
= \frac{D}{C} .
\end{align*}
But then this would imply that $\overrightarrow{\mathbfe{\alpha}}$ has a recurrence relation of length $\degree C +1 < r+1$, which contradicts that $\mathbfe{\alpha} \in \mathcal{L}_n^h (r,r,0)$. Thus, $A_1 (\mathbfe{\alpha})$ and $B$ must be coprime. \\

To prove that the map is injective, we note the following one-to-one correspondences
\begin{align*}
\mathbfe{\alpha}
\longleftrightarrow \overrightarrow{\mathbfe{\alpha}}
\longleftrightarrow L (\overrightarrow{\mathbfe{\alpha}})
\longleftrightarrow \frac{B}{A_1 (\mathbfe{\alpha})} .
\end{align*}

Finally, surjectivity follows from injectivity and the fact that
\begin{align*}
\lvert \mathcal{L}_n^h (r,r,0) \rvert
= (q-1) q^{2r-h-1}
= \big\lvert \big\{ (A,B) \in \mathcal{M}_r \times \mathcal{A}_{<r-h} : (A,B) \text{ coprime} \big\} \big\rvert ,
\end{align*}
where we have used Lemma \ref{lemma, number of sequences of given rhopi form}.
\end{proof}

Let us now make another definition, before proving the final lemma of this section.

\begin{definition} \label{definition, circulant Toeplitz matrix def}
Let $W = (w_0 , w_1 , \ldots , w_s) \in \mathcal{A}_{\leq s}$ (for some $s \geq 0$) and $k \geq 1$. We define the $T_{k+s,k} ([W]_s)$ to be the $(k+s) \times k$ matrix with $j$-th column equal to
\begin{align*}
\begin{pmatrix}
	\smash{\underbrace{\begin{matrix} 0 & \ldots & 0 \end{matrix}}_{\text{ $j-1$ times }}}
	& \smash{\begin{matrix} w_0 & w_1 & \ldots & w_s \end{matrix}}
	& \smash{\underbrace{\begin{matrix} 0 & \ldots & 0 \end{matrix}}_{\text{ $k-j-s$ times }}}
\end{pmatrix}^T . \\
\end{align*}
This is a circulant Toeplitz matrix. 
\end{definition}

\begin{remark} \label{remark, circulant Toeplitz multiplied by vector and by Hankel matrix}
First, we note that for $B \in \mathcal{A}_{\leq k}$, we have $T_{k+s,k} ([W]_s) [B]_k = [WB]_{k+s}$. Now suppose we have an $l \times (k+s)$ Hankel matrix $H_{l , k+s} (\mathbfe{\alpha})$. It is not difficult to see that the matrix $H_{l , k+s} (\mathbfe{\alpha}) T_{k+s,k} ([W]_s)$ is the $l \times k$ Hankel matrix $H_{l,k} (\mathbfe{\alpha} \odot [W]_s)$.
\end{remark}

We are interested in the kernel of $H_{l,k} (\mathbfe{\alpha} \odot [W]_s)$. For $\degree B \leq k$, we have that 
\begin{align*}
&[B]_k \in \kernel H_{l,k} (\mathbfe{\alpha} \odot [W]_s) \\
&\iff H_{l , k+s} (\mathbfe{\alpha}) T_{k+s,k} ([W]_s) [B]_k = \mathbf{0} \\
&\iff H_{l , k+s} (\mathbfe{\alpha}) [WB]_{k+s} = \mathbf{0} .
\end{align*}
That is, there is an injective map from $\kernel H_{l,k} (\mathbfe{\alpha} \odot [W]_s)$ to the subset of $\kernel H_{l , k+s} (\mathbfe{\alpha})$ consisting of vectors $[C]_{k+s}$ for some $C \in \mathcal{A}_{\leq k+s}$ with $W \mid C$. We note that this map is surjective onto the subset if $\degree W = s$, but not necessarily surjective if $\degree W < s$. Indeed, in the latter case, we have $\degree (WB) < k+s$, meaning it does not account for polynomials $C$ with $\degree C = k+s$ and $W \mid C$. The following lemma gives further information on two cases given some restrictions.

\begin{lemma} \label{lemma, U reduction, 1 dimensional kernel case}
Let $s \geq 0$ and let $W \in \mathcal{A}$ with $0 \leq s_1 := \degree W \leq s$. Also, let $\mathbfe{\alpha} \in \mathcal{L}_n (r_1 , \rho_1 , \pi_1)$ with $n \geq 2$ and $n \geq 2r_1 +s-1$, and let $A_1 (\mathbfe{\alpha}) \in \mathcal{M}_{\rho_1}$ be the first characteristic polynomial. Then, 
\begin{align*}
r \Big( \mathbfe{\alpha} \odot [W]_s \Big)
	= &r_1 - \degree ( A_1 (\mathbfe{\alpha}) , W ) - \min \{ s-s_1 , \pi_1 \} , \\
\rho \Big( \mathbfe{\alpha} \odot [W]_s \Big)
	= &\rho_1 - \degree ( A_1 (\mathbfe{\alpha}) , W ) , \\
\pi \Big( \mathbfe{\alpha} \odot [W]_s \Big)
	= &\max \{ 0 , \pi_1 - (s-s_1) \} .
\end{align*}
and 
\begin{align*}
A_1 \Big( \mathbfe{\alpha} \odot [W]_s \Big)
= \frac{A_1 (\mathbfe{\alpha})}{\big( A_1 (\mathbfe{\alpha}) , W \big)} .
\end{align*}

For our second claim, suppose that $s \geq 0$ and let $W \in \mathcal{A}$ with $0 \leq s_1 := \degree W \leq s$. Also, let $\mathbfe{\alpha} \in \strictmathcal{L}_n \Big( \frac{n-s}{2} +1 , 0 , \frac{n-s}{2} +1 \Big)$ with $n-s \geq 2$ being even. Then, $\mathbfe{\alpha} \odot [W]_s \in \strictmathcal{L}_{n-s} \Big( \frac{n-s}{2} +1 , 0 , \frac{n-s}{2} +1 \Big)$.
\end{lemma}

\begin{proof}
Consider the first claim. Let $l,m$ be such that $l+1=r_1$ and $l+m=n$. Then, 
\begin{align} \label{statement, alpha odot [U] kernel, m+1 bound}
m+1 = (n+2) - (l+1) \geq r_1 + s+1 .
\end{align}
Now, we have already established that
\begin{align*}
[B]_{m-s} \in \kernel H_{l+1 , m-s+1} (\mathbfe{\alpha} \odot [W]_s)
\hspace{2em} \iff \hspace{2em}
&[WB]_m \in \kernel H_{l+1 , m+1} (\mathbfe{\alpha}) , \\ 
	&\degree B \leq m-s .
\end{align*}
Since $l+1 = r_1$, we have that
\begin{align*}
\kernel H_{l+1 , m+1} (\mathbfe{\alpha})
= \Big\{ [B_1 A_1 (\mathbfe{\alpha})]_m : B_1 \in \mathcal{A}_{\leq m-r_1 } \Big\} .
\end{align*}
So, we are looking for solutions to the equation
\begin{align*}
B_1 A_1 (\mathbfe{\alpha})
= WB
\end{align*}
where $B_1 \in \mathcal{A}_{\leq m-r_1 }$ and $B \in \mathcal{A}_{m-s}$. We can see that solutions are of the form
\begin{align*}
(B_1 , B)
= \bigg( \frac{W}{\big( A_1 (\mathbfe{\alpha}) , W \big)} C ,
	\frac{A_1 (\mathbfe{\alpha})}{\big( A_1 (\mathbfe{\alpha}) , W \big)} C \bigg)
\end{align*}
for $C \in \mathcal{A}$ with
\begin{align*}
\degree C 
\leq &\min \bigg\{ m-r_1 - \degree \frac{W}{\big( A_1 (\mathbfe{\alpha}) , W \big)} ,
	m-s - \degree \frac{A_1 (\mathbfe{\alpha})}{\big( A_1 (\mathbfe{\alpha}) , W \big)} \bigg\} \\
= &m-r_1 -s + \degree \big( A_1 (\mathbfe{\alpha}) , W \big) + \min \{ s-s_1 , \pi_1 \} .
\end{align*}
Thus,
\begin{align*}
\kernel H_{l+1 , m-s+1} (\mathbfe{\alpha} \odot [W]_s)
= \Bigg\{ \bigg[ C \frac{A_1 (\mathbfe{\alpha})}{\big( A_1 (\mathbfe{\alpha}) , W \big)} \bigg]_{m-s} : \substack{C \in \mathcal{A} \\ \degree C \leq m-r_1 -s + \degree ( A_1 (\mathbfe{\alpha}) , W ) + \min \{ s-s_1 , \pi_1 \} } \Bigg\} ,
\end{align*}
Note that (\ref{statement, alpha odot [U] kernel, m+1 bound}) tells us that $m-r_1 -s + \degree ( A_1 (\mathbfe{\alpha}) , W ) + \min \{ s-s_1 , \pi_1 \} \geq 0$, and so $C$ can take a non-zero value. Therefore, we can see that the first characteristic polynomial is
\begin{align*}
A_1 \Big( \mathbfe{\alpha} \odot [W]_s \Big)
= \frac{A_1 (\mathbfe{\alpha})}{\big( A_1 (\mathbfe{\alpha}) , W \big)} ,
\end{align*}
from which we also deduce that $\rho \Big( \mathbfe{\alpha} \odot [W]_s \Big) = \rho_1 - \degree ( A_1 (\mathbfe{\alpha}) , W )$. The value of $r(\mathbfe{\alpha})$ is simply the number of columns of 
$H_{l+1 , m-s+1} (\mathbfe{\alpha} \odot [W]_s)$ minus the dimension of the kernel, and so we have
\begin{align*}
r \Big( \mathbfe{\alpha} \odot [W]_s \Big)
= &\Big( m-s+1 \Big) - \Big( m-r_1 -s + \degree ( A_1 (\mathbfe{\alpha}) , W ) + \min \{ s-s_1 , \pi_1 \} +1 \Big) \\
= &r_1 - \degree ( A_1 (\mathbfe{\alpha}) , W ) - \min \{ s-s_1 , \pi_1 \} .
\end{align*}
Finally, by definition, we have
\begin{align*}
\pi \Big( \mathbfe{\alpha} \odot [W]_s \Big)
= r \Big( \mathbfe{\alpha} \odot [W]_s \Big) - \rho \Big( \mathbfe{\alpha} \odot [W]_s \Big)
= \max \{ 0 , \pi_1 - (s-s_1) \} .
\end{align*}

For the second claim, we note that the conditions on $\mathbfe{\alpha}$ mean that the first $\frac{n+s}{2}$ entries of $\mathbfe{\alpha}$ are zero, while the $\big( \frac{n+s}{2} + 1 \big)$-th entry is non-zero. So, by definition of $\odot$, we have that there are $\frac{n-s}{2}$ zeros at the start of $\mathbfe{\alpha} \odot [U]_s$, while its $\big( \frac{n-s}{2} + 1 \big)$-th entry is non-zero. From this, we deduce that $\mathbfe{\alpha} \odot [W]_s \in \strictmathcal{L}_{n-s} \Big( \frac{n-s}{2} +1 , 0 , \frac{n-s}{2} +1 \Big)$.
\end{proof}

\section{Proof of Theorem \ref{main theorem, lattice point variance elliptic annuli}} \label{section, main theorem proof}

\begin{proof}[Proof of Theorem \ref{main theorem, lattice point variance elliptic annuli}]

We will prove the theorem for when $n$ is even. When $n$ is odd, the proof is almost identical and we simply swap the roles of $U$ and $V$. Remark \ref{remark, explanation why U,V are odd, even coprime monic} indicates why we consider the two cases separately. Recall that when $n$ is even we have $s:= \degree U$ and $t := \degree V +1$. Also, recall that $s' := \frac{n - s}{2}$ and $t' := \frac{n - t}{2}$; and $n_1 := \Big\lfloor \frac{n+2}{2} \Big\rfloor$ and $n_2 := \Big\lfloor \frac{n+3}{2} \Big\rfloor$. \\

We will consider the three cases of the theorem separately, but first we will reformulate the problem by making use of the Fourier expansion for the indicator function (Definition \ref{definition, Fourier exp indicator function}). We have
\begin{align*}
\frac{1}{q^n} \sum_{A \in \mathcal{M}_n } \bigg\lvert \sum_{B \in I (A; <h)} S_{U,V} (B) \bigg\rvert^2 
= &\frac{4}{q^n} \sum_{A \in \text{$\mathcal{M}_{n}$} }
	\bigg( \sum_{\substack{ B \in \mathcal{M}_{n} \\ \degree (B-A) < h }} 
	\sum_{\substack{E \in \mathcal{M}_{s'} \\ F \in \mathcal{A}_{\leq t'} }} \mathbbm{1} (B - U E^2 - V F^2 ) \bigg)^2 \\
= &\frac{4}{q^n} \sum_{A \in \text{$\mathcal{A}_{\leq n}$} }
	\bigg( \sum_{\substack{B \in \mathcal{A}_{\leq n} \\ \degree (B-A) < h }} 
	\sum_{\substack{E \in \mathcal{M}_{s'} \\ F \in \mathcal{A}_{\leq t'} }} \mathbbm{1} (B - U E^2 - V F^2 ) \bigg)^2 .
\end{align*}
For the first equality, the ranges of $E,F$ are justified by Remark \ref{remark, explanation why U,V are odd, even coprime monic}; and the factor of $4$ appears due to symmetry, to account for the fact that we are not including $-E \in \mathcal{M}_{s'}$ in the summation range. For the second line, note that the conditions on $E,F$, and the equation $B - U E^2 - V F^2 = 0$, force $B$ to be in $\mathcal{M}_n$, and so we have not changed the final result by rewriting the range of $B$ to $\mathcal{A}_{\leq n}$. Similarly, we rewrote the range of $A$ to be in $\mathcal{A}_{\leq n}$, because the fact that $B \in \mathcal{M}_n$, and $\degree (B-A) < h \leq n$, force $A$ to be in $\mathcal{M}_n$. Continuing, we have
\begin{align*}
&\frac{1}{q^n} \sum_{A \in \mathcal{M}_n } \bigg\lvert \sum_{B \in I (A; <h)} S_{U,V} (B) \bigg\rvert^2  \\
= &\frac{4}{q^{3n+2}} \sum_{A \in \mathcal{A}_{\leq n} }
	\bigg( \sum_{\substack{B \in \mathcal{A}_{\leq n} \\ \degree (B-A) < h }} 
	\sum_{\substack{E \in \mathcal{M}_{s'} \\ F \in \mathcal{A}_{\leq t'} }}
	\sum_{\mathbfe{\alpha} \in \mathbb{F}_q^{n+1}} \\
		&\hspace{3em} \psi \Big( \mathbfe{\alpha} \cdot [B]_n
			- [E]_{s'}^T H_{s'+1 , n-s'+1} (\mathbfe{\alpha}) [UE]_{n-s'} 
			- [F]_{t'}^T H_{t'+1 , n-t'+1} (\mathbfe{\alpha}) [VF]_{n-t'} \Big) \bigg)^2 ,
\end{align*}
where we have used the Fourier expansion of $\mathbbm{1} (B - U E^2 - V F^2 )$ given in Definition \ref{definition, Fourier exp indicator function}, and (\ref{statement, how Hankel matrices appear from products}). Note that $\degree UE = s + s' = n-s'$ and $UE$ is monic, and so $[UE]_{n-s'}$ is a vector with final entry equal to $1$. On the other hand, $\degree VE \leq (t-1) + t' = n - t' -1$, and so $[VF]_{n-t'}$ is a vector with at least one zero at the end. Ultimately, this stems from the difference in parity between $n$ and $\degree V$. This will be used later in the proof. \\

Now, if we have $A_1 , A_2$ with $\degree (A_1 - A_2) <h$ then the contributions of the summand when $A=A_1$ and $A=A_2$ are identical. Given that the condition $\degree (A_1 - A_2) <h$ creates equivalence classes of size $q^h$, we can consider one polynomial from each class and multiply by $q^h$. The natural polynomial to take from each class is the one with first $h$ entries being $0$. These polynomials are represented in vector form by $\mathbf{a} \in \{ 0 \}^h \times \mathbb{F}_q^{n-h+1}$. It is not difficult to see that, in vector form, the polynomial $B$ is just $\mathbf{a} + \mathbf{b}$ for some $\mathbf{b} \in \mathbb{F}_q^h \times \{ 0 \}^{n-h+1}$. Thus, we have
\begin{align*}
&\frac{1}{q^n} \sum_{A \in \mathcal{M}_n } \bigg\lvert \sum_{B \in I (A; <h)} S_{U,V} (B) \bigg\rvert^2  \\
= &\frac{4 q^h}{q^{3n+2}} \sum_{\mathbf{a} \in \{ 0 \}^h \times \mathbb{F}_q^{n-h+1} }
	\bigg( \sum_{\mathbf{b} \in \mathbb{F}_q^h \times \{ 0 \}^{n-h+1} }
	\sum_{\substack{E \in \mathcal{M}_{s'} \\ F \in \mathcal{A}_{\leq t'} }}
	\sum_{\mathbfe{\alpha} \in \mathbb{F}_q^{n+1}} \\
		&\hspace{3em} \psi \Big( \mathbfe{\alpha} \cdot (\mathbf{a} + \mathbf{b})
			- [E]_{s'}^T H_{s'+1 , n-s'+1} (\mathbfe{\alpha}) [UE]_{n-s'} 
			- [F]_{t'}^T H_{t'+1 , n-t'+1} (\mathbfe{\alpha}) [VF]_{n-t'} \Big) \bigg)^2 .
\end{align*}
Now, the additive nature of $\psi$ allows us to factor out the term involving $\mathbf{b}$ above. That is, within the large parentheses, we have the factor $\sum_{\mathbf{b} \in \mathbb{F}_q^h \times \{ 0 \}^{n-h+1} } \psi ( \mathbfe{\alpha} \cdot \mathbf{b})$. By the orthogonality relation (\ref{statement, additive character FF orthog relation}), we have
\begin{align*}
\sum_{\mathbf{b} \in \mathbb{F}_q^h \times \{ 0 \}^{n-h+1} } \psi ( \mathbfe{\alpha} \cdot \mathbf{b})
= \prod_{i=0}^{h-1} \sum_{b_i \in \mathbb{F}_q} \psi ( \alpha_i b_i)
= \begin{cases}
q^h &\text{ if $\alpha_0 , \ldots , \alpha_{h-1} = 0$,} \\
0 &\text{ otherwise.}
\end{cases}
\end{align*}
Thus, a non-zero contribution occurs only if $\alpha_0 , \ldots , \alpha_{h-1} = 0$. We can apply a similar approach to handle the term $\mathbfe{\alpha} \cdot \mathbf{a}$. However, unlike with $\mathbf{b}$, the sum over $\mathbf{a}$ appears outside the large parentheses. Whereas, the sum over $\mathbfe{\alpha}$ appears within them. Given the square power, this means that there are ``two $\mathbfe{\alpha}$'' that we must consider, say $\mathbfe{\alpha}_1$ and $\mathbfe{\alpha}_2$. We have already established that the sum over $\mathbf{b}$ ``forces'' the first $h$ entries of $\mathbfe{\alpha}_1$ and $\mathbfe{\alpha}_2$ to be zero. Similar reasoning for sum over $\mathbf{a}$ will ``force'' the condition $\mathbfe{\alpha}_1 + \mathbfe{\alpha}_2 = \mathbf{0}$. That is, $\mathbfe{\alpha}_2 = - \mathbfe{\alpha}_1$. Therefore, we have
\begin{align*}
&\frac{1}{q^n} \sum_{A \in \mathcal{M}_n } \bigg\lvert \sum_{B \in I (A; <h)} S_{U,V} (B) \bigg\rvert^2  \\
= &\frac{4 q^{2h}}{q^{2n+1}} 	\sum_{\mathbfe{\alpha}_1 \in \mathcal{L}_n^h} 
	\bigg( \sum_{\substack{E \in \mathcal{M}_{s'} \\ F \in \mathcal{A}_{\leq t'} }}
		\psi \Big( - [E]_{s'}^T H_{s'+1 , n-s'+1} (\mathbfe{\alpha}_1) [UE]_{n-s'} 
			- [F]_{t'}^T H_{t'+1 , n-t'+1} (\mathbfe{\alpha}_1) [VF]_{n-t'} \Big) \bigg) \\
	&\hspace{4em} \times \bigg( \sum_{\substack{E \in \mathcal{M}_{s'} \\ F \in \mathcal{A}_{\leq t'} }}
		\psi \Big( - [E]_{s'}^T H_{s'+1 , n-s'+1} (-\mathbfe{\alpha}_1) [UE]_{n-s'} 
			- [F]_{t'}^T H_{t'+1 , n-t'+1} (-\mathbfe{\alpha}_1) [VF]_{n-t'} \Big) \bigg) \\
= &\frac{4 q^{2h}}{q^{2n+1}} 	\sum_{\mathbfe{\alpha} \in \mathcal{L}_n^h} 
	\bigg\lvert \sum_{\substack{E \in \mathcal{M}_{s'} \\ F \in \mathcal{A}_{\leq t'} }}
		\psi \Big( [E]_{s'}^T H_{s'+1 , n-s'+1} (\mathbfe{\alpha}) [UE]_{n-s'} 
			+ [F]_{t'}^T H_{t'+1 , n-t'+1} (\mathbfe{\alpha}) [VF]_{n-t'} \Big) \bigg\rvert^2 .
\end{align*}

Now, consider the cases where $\mathbfe{\alpha} \in \mathcal{L}_n^h (0,0,0)$ and $\mathbfe{\alpha} \in \mathcal{L}_n^h (1,0,1)$. This is just when all entries of $\mathbfe{\alpha}$ are zero except the last which can take any value in $\mathbb{F}_q$. It is not difficult to see that the contributions of these cases is $4 q^{2h - \degree U -\degree V +1}$. By (\ref{statement, lattice point ellipse mean value calculations}), this is just
\begin{align*}
\bigg( \frac{1}{q^n} \sum_{A \in \mathcal{M}_n } \sum_{B \in I (A; <h)} S_{U,V} (B) \bigg)^2 ,
\end{align*}
and so
\begin{align}
\begin{split} \label{statement, variance as add char after mean square removed}
&\frac{1}{q^n} \sum_{A \in \mathcal{M}_n } \bigg\lvert \Delta_{S_{U,V}} (A; <h) \bigg\rvert^2 \\
= &\frac{1}{q^n} \sum_{A \in \mathcal{M}_n } \bigg\lvert \sum_{B \in I (A; <h)} S_{U,V} (B) \bigg\rvert^2 
	-\bigg( \frac{1}{q^n} \sum_{A \in \mathcal{M}_n } \sum_{B \in I (A; <h)} S_{U,V} (B) \bigg)^2 \\
= &\frac{4 q^{2h}}{q^{2n+1}} 	\sum_{\substack{\mathbfe{\alpha} \in \mathcal{L}_n^h : \\ \mathbfe{\alpha} \not\in \mathcal{L}_n^h (0,0,0) \\ \mathbfe{\alpha} \not\in \mathcal{L}_n^h (1,0,1) }}
	\bigg\lvert \sum_{\substack{E \in \mathcal{M}_{s'} \\ F \in \mathcal{A}_{\leq t'} }}
		\psi \Big( [E]_{s'}^T H_{s'+1 , n-s'+1} (\mathbfe{\alpha}) [UE]_{n-s'} 
			+ [F]_{t'}^T H_{t'+1 , n-t'+1} (\mathbfe{\alpha}) [VF]_{n-t'} \Big) \bigg\rvert^2 . \\
= &\frac{4 q^{2h}}{q^{2n+1}} 	\sum_{\substack{\mathbfe{\alpha} \in \mathcal{L}_n^h : \\ \mathbfe{\alpha} \not\in \mathcal{L}_n^h (0,0,0) \\ \mathbfe{\alpha} \not\in \mathcal{L}_n^h (1,0,1) }}
	\bigg\lvert \sum_{\substack{E \in \mathcal{M}_{s'} \\ F \in \mathcal{A}_{\leq t'} }}
		\psi \Big( [E]_{s'}^T H_{s'+1 , s'+1} (\mathbfe{\alpha}\odot [U]_s ) [E]_{s'} 
			+ [F]_{t'}^T H_{t'+1 , t'+1} (\mathbfe{\alpha} \odot [V]_t ) [F]_{t'} \Big) \bigg\rvert^2 ,
\end{split}
\end{align}
where the last equality uses Definition \ref{definition, circulant Toeplitz matrix def} and Remark \ref{remark, circulant Toeplitz multiplied by vector and by Hankel matrix}. Specifically,
\begin{align*}
H_{s'+1 , n-s'+1} (\mathbfe{\alpha}) [UE]_{n-s'}
= H_{s'+1 , n-s'+1} (\mathbfe{\alpha}) T_{n-s'+1 , s'+1} ([U]_s) [E]_{s'}
= H_{s'+1 , s'+1} (\mathbfe{\alpha}\odot [U]_s ) [E]_{s'} 
\end{align*}
and
\begin{align*}
H_{t'+1 , n-t'+1} (\mathbfe{\alpha}) [VF]_{n-t'}
= H_{t'+1 , n-t'+1} (\mathbfe{\alpha}) T_{n-t'+1 , t'+1} ([V]_t) [F]_{t'}
= H_{t'+1 , t'+1} (\mathbfe{\alpha} \odot [V]_t ) [F]_{t'}.
\end{align*}

Now that we have established (\ref{statement, variance as add char after mean square removed}), we are in a position to consider each case of the theorem separately. \\

\underline{\textbf{Case 1:}} $h \geq s'+s$. \\

The $\mathbfe{\alpha}$ that appear in (\ref{statement, variance as add char after mean square removed}) are in $\mathcal{L}_n^h$. Remark \ref{remark, rho_1 values dependent on h} tells us that either $h+1 \leq \strictrho (\mathbfe{\alpha} ) \leq n_2 -1$ or $\strictrho (\mathbfe{\alpha}) = 0$. Given that $h \geq s' + s \geq n_2 -1$, we must have that $\strictrho (\mathbfe{\alpha}) = 0$. That is, $\mathbfe{\alpha} \in \strictmathcal{L}_n^h (r_1 , 0 , r_1)$ for some $r_1 \geq 2$ (the cases $r_1 = 0,1$ are excluded in the summation in (\ref{statement, variance as add char after mean square removed})). Remark \ref{remark, rho_1 values dependent on h} now tells us that $H_{n_1 , n_2} (\mathbfe{\alpha})$ is lower skew-triangular and that the first non-zero skew-diagonal will determine the rank and hence the value of $r (\mathbfe{\alpha})$. In particular, the first possible non-zero skew-diagonal is the $(h+1)$-th skew-diagonal, and so
\begin{align*}
r (\mathbfe{\alpha}) \leq (n+1) - h \leq (n+1) - (s' +s) = s'+1.
\end{align*}
Thus, in this case, all $\mathbfe{\alpha}$ that appear in (\ref{statement, variance as add char after mean square removed}) are in $\strictmathcal{L}_n^h (r_1 , 0 , r_1 )$ for some $2 \leq r_1 \leq s'+1$. Lemma \ref{lemma, U reduction, 1 dimensional kernel case} (including the final claim of the lemma) tells us that $\strictpi \Big( \mathbfe{\alpha} \odot [U]_s \Big) = r_1$; and so because $r_1 \geq 2$, Lemma \ref{lemma, quadratic form values over monics and nonmonics} tells us that 
\begin{align*}
\sum_{E \in \mathcal{M}_{s'}} \psi \Big( [E]_{s'}^T H_{s'+1 , s'+1} (\mathbfe{\alpha} \odot [U]_s ) [E]_{s'} \Big)
= 0 .
\end{align*}
Finally, substituting into (\ref{statement, variance as add char after mean square removed}), we have
\begin{align*}
\frac{1}{q^n} \sum_{A \in \mathcal{M}_n } \bigg\lvert \Delta_{S_{U,V}} (A;<h) \bigg\rvert^2 
= 0 .
\end{align*}

\underline{\textbf{Case 2:}} $n_2 -1 \leq h < s'+s$. \\

Again, consider the $\mathbfe{\alpha}$ that appear in (\ref{statement, variance as add char after mean square removed}). As in Case 1, we can show that $\mathbfe{\alpha} \in \strictmathcal{L}_n^h (r_1 , 0 , r_1)$ for some $r_1 \geq 2$; and similar to that case, we have $r_1 \leq (n+1) - h$. We have already established in Case 1 that $r_1 \leq s'+1$ contributes zero. Hence, we only need to consider $\mathbfe{\alpha} \in \strictmathcal{L}_n (r_1 , 0 , r_1 )$ with $s'+1 < r_1 \leq n+1-h$. Also, the second result in Lemma \ref{lemma, quadratic form values over monics and nonmonics} tells us that a non-zero contribution will occur only if $\strictpi (\mathbfe{\alpha} \odot [U]_s ) \leq 1$. So, applying the above to (\ref{statement, variance as add char after mean square removed}), we obtain
\begin{align}
\begin{split} \label{statement, main theorem proof, variance in terms of kernels, first}
&\frac{1}{q^n} \sum_{A \in \mathcal{M}_n } \bigg\lvert \Delta_{S_{U,V}} (A;<h) \bigg\rvert^2 \\
= &\frac{4 q^{2h}}{q^{2n+1}} \sum_{r_1 = s'+2}^{n+1-h}
	\sum_{\substack{\mathbfe{\alpha} \in \strictmathcal{L}_n (r_1,0,r_1) : \\ \strictpi (\mathbfe{\alpha} \odot [U]_s ) \leq 1 }} \\
	&\hspace{3em} \bigg\lvert \sum_{\substack{E \in \mathcal{M}_{s'} \\ F \in \mathcal{A}_{\leq t'} }}
		\psi \Big( [E]_{s'}^T H_{s'+1 , s'+1} (\mathbfe{\alpha}\odot [U]_s ) [E]_{s'} 
			+ [F]_{t'}^T H_{t'+1 , t'+1} (\mathbfe{\alpha} \odot [V]_t ) [F]_{t'} \Big) \bigg\rvert^2 \\
= &\frac{4 q^{2h}}{q^{2n+1}} \sum_{r_1 = s'+2}^{n+1-h}
	\sum_{\substack{\mathbfe{\alpha} \in \strictmathcal{L}_n (r_1,0,r_1) : \\ \strictpi (\mathbfe{\alpha} \odot [U]_s ) \leq 1 }} 
		q^{s' + t' + \strictpi (\mathbfe{\alpha} \odot [U]_s ) }  
			\Big\lvert \kernel H_{s'+1 , s'+1} (\mathbfe{\alpha} \odot [U]_s ) \Big\rvert
			\Big\lvert \kernel H_{t'+1 , t'+1} (\mathbfe{\alpha} \odot [V]_t ) \Big\rvert ,
\end{split}
\end{align}
where the second equality follows by Lemma \ref{lemma, quadratic form values over monics and nonmonics}. Now, let $\mathbfe{\alpha}'$ be the sequence we obtain by removing the last term of $\mathbfe{\alpha}$. It will be helpful to reformulate the above in terms of $\mathbfe{\alpha}'$ instead of $\mathbfe{\alpha}$. Note that $\mathbfe{\alpha}' \odot [U]_s$ is the same sequence as the one obtained by removing the last term of $\mathbfe{\alpha} \odot [U]_s$. Furthermore, it is not difficult to see that, due to the zero at the end of $[V]_t$, we have $\mathbfe{\alpha} \odot [V]_t = \mathbfe{\alpha}' \odot [V]_{t-1}$ and thus $H_{t'+1 , t'+1} (\mathbfe{\alpha} \odot [V]_t ) = H_{t'+1 , t'+1} (\mathbfe{\alpha}' \odot [V]_{t-1} )$. We can also use Remark \ref{lemma, effect of removing last entry of alpha} to see that
\begin{align*}
\Big\lvert \kernel H_{s'+1 , s'+1} (\mathbfe{\alpha} \odot [U]_s ) \Big\rvert
= \begin{cases}
\Big\lvert \kernel H_{s' , s'+1} (\mathbfe{\alpha}' \odot [U]_s ) \Big\rvert &\text{ if $\strictpi (\mathbfe{\alpha} \odot [U]_n ) = 0$,} \\
q^{-1} \Big\lvert \kernel H_{s' , s'+1} (\mathbfe{\alpha}' \odot [U]_s ) \Big\rvert &\text{ if $\strictpi (\mathbfe{\alpha} \odot [U]_n ) = 1$.} 
\end{cases}
\end{align*}
Finally, Remark \ref{lemma, effect of removing last entry of alpha} also tells us that $\mathbfe{\alpha}' \in \strictmathcal{L}_{n-1} (r_1 -1,0,r_1 -1)$. Applying these points to (\ref{statement, main theorem proof, variance in terms of kernels, first}), we obtain
\begin{align}
\begin{split} \label{statement, main theorem, variance in terms of kernels of shortened alpha}
&\frac{1}{q^n} \sum_{A \in \mathcal{M}_n } \bigg\lvert \Delta_{S_{U,V}} (A;<h) \bigg\rvert^2 \\
= &\frac{4 q^{2h+s' + t'} }{q^{2n+1}} \sum_{r_1 = s'+2}^{n+1-h}
	\sum_{\substack{\mathbfe{\alpha}' \in \strictmathcal{L}_{n-1} (r_1 -1,0,r_1 -1) : \\ \strictpi (\mathbfe{\alpha}' \odot [U]_s ) = 0 \\ a_n \in \mathbb{F}_q }} 
			\Big\lvert \kernel H_{s' , s'+1} (\mathbfe{\alpha}' \odot [U]_s ) \Big\rvert
			\Big\lvert \kernel H_{t'+1 , t'+1} (\mathbfe{\alpha}' \odot [V]_{t-1} ) \Big\rvert \\
= &\frac{4 q^{2h+s' + t'} }{q^{2n}} \sum_{r_1 = s'+1}^{n-h}
	\sum_{\substack{\mathbfe{\alpha} \in \strictmathcal{L}_{n-1} (r_1,0,r_1) : \\ \strictpi (\mathbfe{\alpha} \odot [U]_s ) = 0 }} 
			\Big\lvert \kernel H_{s' , s'+1} (\mathbfe{\alpha} \odot [U]_s ) \Big\rvert
			\Big\lvert \kernel H_{t'+1 , t'+1} (\mathbfe{\alpha} \odot [V]_{t-1} ) \Big\rvert .
\end{split}
\end{align}
For the second equality, the summand is independent of $a_n$ and so we can remove the sum by multiplying the entire expression by $q$. We also relabelled $\mathbfe{\alpha}'$ to $\mathbfe{\alpha}$ for presentational reasons. Now, we have that
\begin{align}
\begin{split} \label{statement, main theorem, kernels expressed as polynomial equations}
&\sum_{\substack{\mathbfe{\alpha} \in \strictmathcal{L}_{n-1} (r_1,0,r_1) : \\ \strictpi (\mathbfe{\alpha} \odot [U]_s ) = 0 }} 
			\Big\lvert \kernel H_{s' , s'+1} (\mathbfe{\alpha} \odot [U]_s ) \Big\rvert
			\Big\lvert \kernel H_{t'+1 , t'+1} (\mathbfe{\alpha} \odot [V]_{t-1} ) \Big\rvert \\
= &\frac{q-1}{q^{n-2r_1 +2}} \sum_{A_2 (\mathbfe{\alpha}) \in \mathcal{M}_{n-r_1+1}}
	\bigg( q \sum_{\substack{B_1 \in \mathcal{A}_{\leq s'+s-r_1} \\ B_2 \in \mathcal{M}_{r_1 -s'-1} \\ B \in \mathcal{M}_{s'} \\ B_1 + B_2 A_2 (\mathbfe{\alpha}) = UB}} 1 \bigg) 
	\bigg( \sum_{\substack{C_1 \in \mathcal{A}_{\leq t'+t-r_1 -1} \\ C_2 \in \mathcal{A}_{\leq r_1 -t'-2} \\ C \in \mathcal{A}_{\leq t'} \\ C_1 + C_2 A_2 (\mathbfe{\alpha}) = VC}} 1 \bigg) \\
= &\frac{q-1}{q^{n-2r_1 +2}} \sum_{A \in \mathcal{M}_{n-r_1+1}}
	\bigg( q \sum_{\substack{B_1 \in \mathcal{A}_{\leq s'+s-r_1} \\ B_2 \in \mathcal{M}_{r_1 -s'-1} \\ B \in \mathcal{M}_{s'} \\ B_1 + B_2 A = UB}} 1 \bigg) 
	\bigg( \sum_{\substack{C_1 \in \mathcal{A}_{\leq t'+t-r_1 -1} \\ C_2 \in \mathcal{A}_{\leq r_1 -t'-2} \\ C \in \mathcal{A}_{\leq t'} \\ C_1 + C_2 A = VC}} 1 \bigg) .
\end{split}
\end{align}
For the second equality we simply rewrote $A_2 (\mathbfe{\alpha})$ as $A$, for presentational reasons. The justification for the first equality is as follows. First, we wish to express the sum over\break $\mathbfe{\alpha} \in \strictmathcal{L}_{n-1} (r_1,0,r_1)$ as a sum over characteristic polynomials. Lemma \ref{lemma, kernel of Hankel matrices} tells us that that $A_1 (\mathbfe{\alpha}) \in \mathcal{M}_{\rho (\mathbfe{\alpha})}$. Since $\rho (\mathbfe{\alpha}) = 0$, we have that $A_1 (\mathbfe{\alpha}) = 1$. Lemma \ref{lemma, kernel of Hankel matrices} also tells us that $A_2 (\mathbfe{\alpha}) \in \mathcal{M}_{n-r_1+1}$ (Note that in the Lemma we have $\mathbfe{\alpha} = (\alpha_0 , \ldots , \alpha_{n})$ while here we have $\mathbfe{\alpha} = (\alpha_0 , \ldots , \alpha_{n-1})$). Given these two points, we may wish to replace $\mathbfe{\alpha} \in \strictmathcal{L}_{n-1} (r_1,0,r_1)$ with $A_2 (\mathbfe{\alpha}) \in \mathcal{M}_{n-r_1+1}$. However, the correspondence between $\mathbfe{\alpha}$ and $A_2 (\mathbfe{\alpha})$ here is not bijective. Indeed, Remark \ref{remark, kernel of Hankel matrices} tells us that $A_2 (\mathbfe{\alpha})$ is unique only up to addition of polynomials of the form $B_1 A_1 (\mathbfe{\alpha})$ for $B_1 \in \mathcal{A}_{\leq n-2r_1 +1}$. Remark \ref{remark, kernel of Hankel matrices} also tells us that for each pair of characteristic polynomials $A_1 , A_2$, there are $q-1$ associated Hankel matrices (simply because multiplying our matrices by an element in $\mathbb{F}_q^*$ does not change the characteristic polynomials). Thus, we can replace $\mathbfe{\alpha} \in \strictmathcal{L}_{n-1} (r_1,0,r_1)$ with $A_2 (\mathbfe{\alpha}) \in \mathcal{M}_{n-r_1+1}$, but we must multiply by the factor $\frac{(q-1)}{\lvert \mathcal{A}_{\leq n-2r_1 +1} \rvert} = \frac{q-1}{q^{n-2r_1 +2}}$. However, this requires one last piece of justification. Namely, that our summand is independent of which representation of $A_2 (\mathbfe{\alpha})$ one takes. This ultimately follows from the conditions $B_1 + B_2 A_2 (\mathbfe{\alpha}) = UB$ and $C_1 + C_2 A_2 (\mathbfe{\alpha}) = VC$, and the ranges given for $B_1 , B_2 , B , C_1 , C_2 , C$. This justifies the change in the summation range. Let us now justify the change in the summand. \\

Consider $\Big\lvert \kernel H_{s' , s'+1} (\mathbfe{\alpha} \odot [U]_s ) \Big\rvert$. Similarly as in the proof of Lemma \ref{lemma, U reduction, 1 dimensional kernel case}, we have that a polynomial $B \in \mathcal{A}_{\leq s'}$ satisfies $[B]_{s'} \in \kernel H_{s' , s'+1} (\mathbfe{\alpha} \odot [U]_s )$ if and only if $UB = B_1 + B_2 A_2 (\mathbfe{\alpha})$ for some $B_1 \in \mathcal{A}_{\leq s'+s-r_1}$ and $B_2 \in \mathcal{A}_{\leq r_1 -s'-1}$. However, recall that we also have the condition $\strictpi (\mathbfe{\alpha} \odot [U]_s ) = 0$. Remark \ref{remark, pi_1 = 0 iff vector in kernel with no zero at end} tells us that $\strictpi (\mathbfe{\alpha} \odot [U]_s ) = 0$ if and only if there is $B \in \mathcal{M}_{s'}$ with $[B]_{s'} \in \kernel H_{s' , s'+1} (\mathbfe{\alpha} \odot [U]_s )$. Note that if $B \in \mathcal{M}_{s'}$, then the condition $B_1 + B_2 A_2 (\mathbfe{\alpha}) = UB$ will require that $B_2 \in \mathcal{M}_{r_1 -s'-1}$. So we have established that the sum
\begin{align*}
\sum_{\substack{B_1 \in \mathcal{A}_{\leq s'+s-r_1} \\ B_2 \in \mathcal{M}_{r_1 -s'-1} \\ B \in \mathcal{M}_{s'} \\ B_1 + B_2 A_2 (\mathbfe{\alpha}) = UB}} 1
\end{align*}
will be zero if $\strictpi (\mathbfe{\alpha} \odot [U]_s ) \geq 0$, as required. Although, when $\strictpi (\mathbfe{\alpha} \odot [U]_s ) = 0$, because of the additional restrictions on $B$ and $B_2$ to be monic, the sum is not actually equal to $\Big\lvert \kernel H_{s' , s'+1} (\mathbfe{\alpha} \odot [U]_s ) \Big\rvert$ which is what we originally wanted. However, Remark \ref{remark, pi_1 = 0 iff vector in kernel with no zero at end} also tells us that this can be remedied by multiplying by $q$. Finally, $\Big\lvert \kernel H_{t'+1 , t'+1} (\mathbfe{\alpha} \odot [V]_{t-1} ) \Big\rvert$ can be express as a similar sum, but it is simpler because we do not have a condition analogous to $\strictpi (\mathbfe{\alpha} \odot [U]_s ) = 0$. \\

Continuing, by substituting (\ref{statement, main theorem, kernels expressed as polynomial equations}) into (\ref{statement, main theorem, variance in terms of kernels of shortened alpha}), we obtain
\begin{align}
\begin{split} \label{statement, main theorem, after kernels expressed as polynomial equations}
&\frac{1}{q^n} \sum_{A \in \mathcal{M}_n } \bigg\lvert \Delta_{S_{U,V}} (A;<h) \bigg\rvert^2 \\
= &4 (q-1) \frac{q^{2h+s' + t'} }{q^{3n+1}} \sum_{r_1 = s'+1}^{n-h} q^{2 r_1}
	\sum_{A \in \mathcal{M}_{n-r_1+1}}
	\bigg( \sum_{\substack{B_1 \in \mathcal{A}_{\leq s'+s-r_1} \\ B_2 \in \mathcal{M}_{r_1 -s'-1} \\ B \in \mathcal{M}_{s'} \\ B_1 + B_2 A = UB}} 1 \bigg) 
	\bigg( \sum_{\substack{C_1 \in \mathcal{A}_{\leq t'+t-r_1 -1} \\ C_2 \in \mathcal{A}_{\leq r_1 -t'-2} \\ C \in \mathcal{A}_{\leq t'} \\ C_1 + C_2 A = VC}} 1 \bigg) .
\end{split}
\end{align}

We wish to use Dirichlet characters\footnote{See Section 1.4 of \cite{Yiasemides2020_PhDThesis_PostMinorCorr} for details on Dirichlet characters in function fields, including orthogonality relations.} of conductor $U$ to address the condition $B_2 A = UB - B_1$. To do so, we require $(B_1 , U)=1$. This is clearly not always the case, and so we will condition on the value of $(B_1 , U)$. Suppose $(B_1 , U) = U'$ and let $U_3$ be such that $U' U_3 = U$. First note that we require $\degree U' \leq \degree B_1 \leq s'+s-r_1$. Second, the equation $B_2 A = UB - B_1$ forces $\big( B_2 A , U \big)=U'$, and so we can write $U' = U_1 U_2$, where $(B_2 , U)=U_1$ and $U_2 \mid A$. Note that we require $\degree U_1 \leq \degree B_2 = r_1 -s' -1$. Applying the above, as well as a similar approach for the sum involving $C_1 , C_2 , C$, we obtain
\begin{align}
\begin{split} \label{statement, main theorem, just before dirichlet character application}
&\sum_{A \in \mathcal{M}_{n-r_1+1}}
	\bigg( \sum_{\substack{B_1 \in \mathcal{A}_{\leq s'+s-r_1} \\ B_2 \in \mathcal{M}_{\leq r_1 -s'-1} \\ B \in \mathcal{M}_{s'} \\ B_1 + B_2 A = UB}} 1 \bigg) 
		\bigg( \sum_{\substack{C_1 \in \mathcal{A}_{\leq t'+t-r_1 -1} \\ C_2 \in \mathcal{A}_{\leq r_1 -t'-2} \\ C \in \mathcal{A}_{\leq t'} \\ C_1 + C_2 A = VC}} 1 \bigg) \\
= & \sum_{\substack{U_1 \in \mathcal{M}_{\leq r_1-s'-1} \\U_2 \in \mathcal{M}_{\leq s'+s-r_1-\degree U_1} \\U_3 \in \mathcal{M} \\ U_1 U_2 U_3 = U }}
	\sum_{\substack{V_1 \in \mathcal{M}_{\leq r_1-t'-2} \\ V_2 \in \mathcal{M}_{\leq t'+t-r_1 -1-\degree V_1} \\ V_3 \in \mathcal{M} \\ V_1 V_2 V_3 = V }}
	\sum_{\substack{A \in \mathcal{M}_{n-r_1+1} \\ U_2 \mid A \\ V_2 \mid A}}
	\bigg( \sum_{\substack{B_1 \in \mathcal{A}_{\leq s'+s-r_1} \\ (B_1 , U) = U_1 U_2 \\ B_2 \in \mathcal{M}_{\leq r_1 -s'-1} \\ (B_2 , U) = U_1 \\ B \in \mathcal{M}_{s'} \\ B_1 + B_2 A = UB }} 1 \bigg)
	\bigg( \sum_{\substack{C_1 \in \mathcal{A}_{\leq t'+t-r_1 -1} \\ (C_1 , V) = V_1 V_2 \\ C_2 \in \mathcal{A}_{\leq r_1 -t' -2} \\ (C_2 , V) = V_1 \\ C \in \mathcal{A}_{\leq t'} \\ C_1 + C_2 A = VC }} 1 \bigg) \\
= & \sum_{\substack{U_1 \in \mathcal{M}_{\leq r_1-s'-1} \\U_2 \in \mathcal{M}_{\leq s'+s-r_1-\degree U_1} \\U_3 \in \mathcal{M} \\ U_1 U_2 U_3 = U }}
	\sum_{\substack{V_1 \in \mathcal{M}_{\leq r_1-t'-2} \\ V_2 \in \mathcal{M}_{\leq t'+t-r_1 -1-\degree V_1} \\ V_3 \in \mathcal{M} \\ V_1 V_2 V_3 = V }}
	\sum_{\substack{B_1 \in \mathcal{A}_{\leq s'+s-r_1-\degree U_1 U_2} \\ (B_1 , U_3) = 1 \\ B_2 \in \mathcal{M}_{\leq r_1 -s'-1-\degree U_1} \\ (B_2 , U_2 U_3) = 1 }} 
	\sum_{\substack{C_1 \in \mathcal{A}_{\leq t'+t-r_1 -1-\degree V_1 V_2} \\ (C_1 , V_3) = 1 \\ C_2 \in \mathcal{A}_{\leq r_1 -t' -2-\degree V_1} \\ (C_2 , V_2 V_3) = 1 }} 
	\sum_{\substack{A \in \mathcal{M}_{n-r_1+1-\degree U_2 V_2} \\ B \in \mathcal{M}_{s'} \\ C \in \mathcal{A}_{\leq t'} \\ B_1 + B_2 V_2 A = U_3 B \\ C_1 + C_2 U_2 A = V_3 C }} 1 .
\end{split}
\end{align}
The last summation can be addressed using Dirichlet characters. In what follows, for a general monic polynomial $W$, a sum over all Dirichlet characters of conductor $W$ is expressed as $\sum_{\chi \modulus W}$. We have,
\begin{align}
\begin{split} \label{statement, main theorem proof, final application of Dirichlet characters}
&\sum_{\substack{A \in \mathcal{M}_{n-r_1+1-\degree U_2 V_2} \\ B \in \mathcal{M}_{s'} \\ C \in \mathcal{A}_{\leq t'} \\ B_1 + B_2 V_2 A = U_3 B \\ C_1 + C_2 U_2 A = V_3 C }} 1 \\
= &\sum_{A \in \mathcal{M}_{n-r_1+1-\degree U_2 V_2}}
	\bigg( \frac{1}{\phi (U_3)} \sum_{\chi_1 \modulus U_3} \chi_1 (B_2 V_2 A) \conj\chi_1 (-B_1) \bigg)
	\bigg( \frac{1}{\phi (V_3)} \sum_{\chi_2 \modulus V_3} \chi_2 (C_2 U_2 A) \conj\chi_2 (-C_1) \bigg) \\
= &\frac{1}{\phi (U_3)} \frac{1}{\phi (V_3)} 
	\sum_{\substack{\chi_1 \modulus U_3 \\ \chi_2 \modulus V_3 }} 
	\chi_1 (B_2 V_2 ) \conj\chi_1 (-B_1) \chi_2 (C_2 U_2 ) \conj\chi_2 (-C_1)
	\sum_{A \in \mathcal{M}_{n-r_1+1-\degree U_2 V_2}} \chi_1 (A) \chi_2 (A) \\
= &\frac{q^{n-r_1+1}}{\lvert U_2 U_3 V_2 V_3 \rvert} ,
\end{split}
\end{align}
where, for the first equality we used the orthogonality relations. For the final equality, we used the fact that $\chi_1 \chi_2$ is a Dirichlet character of conductor $U_3 V_3$ (since $U_3 ,V_3$ are coprime), and so due to the orthogonality relations and the fact that $n-r_1 +1 - \degree U_2 V_2 > \degree U_3 V_3$, we can see that a non-zero contribution occurs only if $\chi_1 \chi_2$ is trivial (and thus $\chi_1 , \chi_2$ are trivial\footnote{To prove this, without obtaining further results on Dirichlet characters, one can prove the contrapositive statement, that if $\chi_1 , \chi_2$ are not both trivial, then $\chi_1 \chi_2$ is not trivial. Indeed, suppose $\chi_1$ is non-trivial and let $D$ be such that $ \chi_1 (D) \neq 0,1$. Since $(U_3 , V_3)=1$, there exists a polynomial $E$ such that $E \equiv D (\modulus U_3)$ and $E \equiv 1 (\modulus V_3)$. Thus, $\chi_1 \chi_2 (E) = \chi_2 (E) \neq 0,1$, and so $\chi_1 \chi_2$ is not trivial.}). Note that our result holds true even for the special cases where $U_3 = 1$ or $V_3 = 1$, including the special subcases where $B_1 = 0$ or $C_1 = 0$. Let us now apply (\ref{statement, main theorem proof, final application of Dirichlet characters}) to (\ref{statement, main theorem, just before dirichlet character application}) and consider the sum over $U_1 , U_2 , U_3$. A similar result applies for the sum over $V_1, V_2 , V_3$. We have
\begin{align*}
\sum_{\substack{U_1 \in \mathcal{M}_{\leq r_1-s'-1} \\U_2 \in \mathcal{M}_{\leq s'+s-r_1-\degree U_1} \\U_3 \in \mathcal{M} \\ U_1 U_2 U_3 = U }}
	\sum_{\substack{B_1 \in \mathcal{A}_{\leq s'+s-r_1-\degree U_1 U_2} \\ (B_1 , U_3) = 1 \\ B_2 \in \mathcal{M}_{\leq r_1 -s'-1-\degree U_1} \\ (B_2 , U_2 U_3) = 1 }} 
	\frac{1}{\lvert U_2 U_3 \rvert} 
= &\frac{1}{\lvert U \rvert} \sum_{\substack{U_1 \in \mathcal{M}_{\leq r_1-s'-1} \\U_2 \in \mathcal{M}_{\leq s'+s-r_1-\degree U_1} \\U_3 \in \mathcal{M} \\ U_1 U_2 U_3 = U }}
	\sum_{\substack{B_1 \in \mathcal{A}_{\leq s'+s-r_1} \\ (B_1 , U) = U_1 U_2 }}
	\sum_{\substack{B_2 \in \mathcal{M}_{\leq r_1 -s'-1} \\ (B_2 , U) = U_1 }} \lvert U_1 \rvert \\
= &\frac{1}{\lvert U \rvert} \sum_{U' U_3 = U}
	\sum_{\substack{B_1 \in \mathcal{A}_{\leq s'+s-r_1} \\ (B_1 , U) = U' }}
	\sum_{U_1 U_2 = U'}
	\sum_{\substack{B_2 \in \mathcal{M}_{\leq r_1 -s'-1} \\ (B_2 , U) = U_1 }} \lvert (B_2 , U) \rvert \\
= &\frac{1}{\lvert U \rvert} \sum_{U' U_3 = U}
	\sum_{\substack{B_1 \in \mathcal{A}_{\leq s'+s-r_1} \\ B_2 \in \mathcal{M}_{\leq r_1 -s'-1} \\ (B_1 , U) = U' \\ (B_2 , U) \mid (B_1 , U) }} \lvert (B_2 , U) \rvert \\
= &\frac{1}{\lvert U \rvert} 	\sum_{\substack{B_1 \in \mathcal{A}_{\leq s'+s-r_1} \\ B_2 \in \mathcal{M}_{\leq r_1 -s'-1} \\ (B_2 , U) \mid B_1 }} \lvert (B_2 , U) \rvert . 
\end{align*}
Similarly,
\begin{align*}
\sum_{\substack{V_1 \in \mathcal{M}_{\leq r_1-t'-2} \\ V_2 \in \mathcal{M}_{\leq t'+t-r_1 -1-\degree V_1} \\ V_3 \in \mathcal{M} \\ V_1 V_2 V_3 = V }}
	\sum_{\substack{C_1 \in \mathcal{A}_{\leq t'+t-r_1 -1-\degree V_1 V_2} \\ (C_1 , V_3) = 1 \\ C_2 \in \mathcal{A}_{\leq r_1 -t' -2-\degree V_1} \\ (C_2 , V_2 V_3) = 1 }}
	\frac{1}{\lvert V_2 V_3 \rvert} 
= \frac{1}{\lvert V \rvert} \sum_{\substack{C_1 \in \mathcal{A}_{\leq t'+t-r_1 -1} \\ C_2 \in \mathcal{A}_{\leq r_1 -t'-2} \\ (C_2 , V) \mid C_1 }} \lvert (C_2 , V) \rvert . 
\end{align*}
Hence, by applying these to (\ref{statement, main theorem, just before dirichlet character application}), we obtain
\begin{align*}
&\sum_{A_2 (\mathbfe{\alpha}) \in \mathcal{M}_{n-r_1+1}}
	\bigg( \sum_{\substack{B_1 \in \mathcal{A}_{\leq s'+s-r_1} \\ B_2 \in \mathcal{M}_{\leq r_1 -s'-1} \\ B \in \mathcal{M}_{s'} \\ B_1 + B_2 A_2 (\mathbfe{\alpha}) = UB}} 1 \bigg) 
		\bigg( \sum_{\substack{C_1 \in \mathcal{A}_{\leq t'+t-r_1 -1} \\ C_2 \in \mathcal{A}_{\leq r_1 -t'-2} \\ C \in \mathcal{A}_{\leq t'} \\ C_1 + C_2 A_2 (\mathbfe{\alpha}) = VC}} 1 \bigg) \\
= &\frac{q^{n-r_1+1}}{\lvert UV \rvert}
		\sum_{\substack{B_1 \in \mathcal{A}_{\leq s'+s-r_1} \\ B_2 \in \mathcal{M}_{\leq r_1 -s'-1} \\ (B_2 , U) \mid B_1 }} \lvert (B_2 , U) \rvert
		\sum_{\substack{C_1 \in \mathcal{A}_{\leq t'+t-r_1 -1} \\ C_2 \in \mathcal{A}_{\leq r_1 -t'-2} \\ (C_2 , V) \mid C_1 }} \lvert (C_2 , V) \rvert .
\end{align*}
Finally, applying this to (\ref{statement, main theorem, after kernels expressed as polynomial equations}), we obtain
\begin{align*}
&\frac{1}{q^n} \sum_{A \in \mathcal{M}_n } \bigg\lvert \Delta_{S_{U,V}} (A;<h) \bigg\rvert^2 \\
= &4 q^{h} \frac{q-1}{q^{\frac{1}{2}} \lvert UV \rvert^{\frac{1}{2}} } \sum_{r_1 = s'+1}^{n-h} q^{r_1 -(n-h)}
		\sum_{\substack{B_1 \in \mathcal{A}_{\leq s'+s-r_1} \\ B_2 \in \mathcal{M}_{\leq r_1 -s'-1} \\ (B_2 , U) \mid B_1 }} \lvert (B_2 , U) \rvert
		\sum_{\substack{C_1 \in \mathcal{A}_{\leq t'+t-r_1 -1} \\ C_2 \in \mathcal{A}_{\leq r_1 -t'-2} \\ (C_2 , V) \mid C_1 }} \lvert (C_2 , V) \rvert \\
= & q^h f_{U,V} (n,h) .
\end{align*}

Let us now prove the bound for $f_{U,V} (n,h)$. We have
\begin{align} \label{statement, main theorem, case 3, what must be bounded from case 2}
f_{U,V} (n,h)
= \frac{4(q-1)}{q^{\frac{1}{2}} \lvert UV \rvert^{\frac{1}{2}} } \sum_{r_1 = s'+1}^{n-h} q^{r_1 -(n-h)}
		\sum_{\substack{B_1 \in \mathcal{A}_{\leq s'+s-r_1} \\ B_2 \in \mathcal{M}_{\leq r_1 -s'-1} \\ (B_2 , U) \mid B_1 }} \lvert (B_2 , U) \rvert
		\sum_{\substack{C_1 \in \mathcal{A}_{\leq t'+t-r_1 -1} \\ C_2 \in \mathcal{A}_{\leq r_1 -t'-2} \\ (C_2 , V) \mid C_1 }} \lvert (C_2 , V) \rvert .
\end{align}
We note that $(B_2 , U) \mid B_1$ and $(B_2 , U) \mid B_2$, and so
\begin{align*}
\degree (B_2 , U)
\leq \min \{ \degree B_1 , \degree B_2 \}
\leq \min \{ s'+s-r_1 , r_1 -s'-1 \}
\leq \frac{\degree U}{2} -1 .
\end{align*}
Similarly, $\degree (C_2 , V) \leq \frac{\degree V -1}{2} -1$. So, we have 
\begin{align*}
&\sum_{\substack{B_1 \in \mathcal{A}_{\leq s'+s-r_1} \\ B_2 \in \mathcal{M}_{\leq r_1 -s'-1} \\ (B_2 , U) \mid B_1}} \lvert (B_2 , U) \rvert
= \sum_{\substack{U_1 \mid U \\ \degree U_1 \leq \frac{\degree U}{2} -1 }} \lvert U_1 \rvert
	\sum_{\substack{B_2 \in \mathcal{M}_{\leq r_1 -s'-1} \\ (B_2 , U) = U_1 }}
	\sum_{\substack{B_1 \in \mathcal{A}_{\leq s'+s-r_1} \\ U_1 \mid B_1 }} 1 \\
= &q^{s'+s-r_1 +1} \sum_{\substack{U_1 \mid U \\ \degree U_1 \leq \frac{\degree U}{2} -1 }}
	\sum_{\substack{B_2 \in \mathcal{M}_{\leq r_1 -s'-1} \\ (B_2 , U) = U_1 }} 1 
\leq q^{s} \sum_{\substack{U_1 \mid U \\ \degree U_1 \leq \frac{\degree U}{2} -1 }} \frac{1}{ \lvert U_1 \rvert} 
\leq \lvert U \rvert (\log_q \degree U) .
\end{align*}
Similarly, 
\begin{align*}
\sum_{\substack{C_1 \in \mathcal{A}_{\leq t'+t-r_1 -1} \\ C_2 \in \mathcal{A}_{\leq r_1 -t'-2} \\ (C_2 , U) \mid C_1 }} \lvert (C_2 , V) \rvert
\leq \lvert V \rvert (\log_q \degree V) .
\end{align*}
Applying this to (\ref{statement, main theorem, case 3, what must be bounded from case 2}) we obtain
\begin{align*}
&\frac{4(q-1)}{q^{\frac{1}{2}} \lvert UV \rvert^{\frac{1}{2}} } \sum_{r_1 = s'+1}^{n-h} q^{r_1 -(n-h)}
		\sum_{\substack{B_1 \in \mathcal{A}_{\leq s'+s-r_1} \\ B_2 \in \mathcal{M}_{\leq r_1 -s'-1} \\ (B_2 , U) \mid B_1 }} \lvert (B_2 , U) \rvert
		\sum_{\substack{C_1 \in \mathcal{A}_{\leq t'+t-r_1 -1} \\ C_2 \in \mathcal{A}_{\leq r_1 -t'-2} \\ (C_2 , V) \mid C_1 }} \lvert (C_2 , V) \rvert \\
\leq & 4 (q-1) q^{- \frac{1}{2}} \Big( \sum_{r_1 = s'+1}^{n-h} q^{r_1 - (n-h)}  \Big)
	\lvert UV \rvert^{\frac{1}{2}} 
	(\log_q \degree U) 
	(\log_q \degree V) \\
\leq &4 q^{\frac{1}{2}}
	\lvert UV \rvert^{\frac{1}{2}} 
	(\log_q \degree U)
	(\log_q \degree V) . \\
\end{align*}

\underline{\textbf{Case 3:}} $3 (\degree UV +1) \leq h < \min \{ s' , t' \} -1$. \\

As in Cases 1 and 2, consider the $\mathbfe{\alpha} \in \mathcal{L}_n^h$ that appear in (\ref{statement, variance as add char after mean square removed}). Remark \ref{remark, rho_1 values dependent on h} tells us that either $h+1 \leq \strictrho (\mathbfe{\alpha} ) \leq n_2 -1$ or $\strictrho (\mathbfe{\alpha}) = 0$. Certainly, as in Cases 1 and 2, $\strictrho (\mathbfe{\alpha}) = 0$ is a possibility; but now, we also have that the range $h+1 \leq \strictrho (\mathbfe{\alpha} ) \leq n_2 -1$ is non-empty. We will divide this range into two parts: The $\mathbfe{\alpha}$ that satisfy
\begin{align} \label{statement, main theorem, LOT second contribution range}
h+1 \leq \strictrho (\mathbfe{\alpha} ) \leq n_2 -1
\hspace{3em} \text{ and } \hspace{3em}
r (\mathbfe{\alpha}) \geq \min \{ s' , t' \} +1 ;
\end{align}
and the $\mathbfe{\alpha}$ that satisfy
\begin{align} \label{statement, main theorem, main term contribution range}
h+1 \leq \strictrho (\mathbfe{\alpha} ) \leq \min \{ s' , t' \}
\hspace{3em} \text{ and } \hspace{3em}
r (\mathbfe{\alpha}) \leq \min \{ s' , t' \} .
\end{align}

Step 1 below will address the error terms, which come from the cases $\strictrho (\mathbfe{\alpha}) = 0$ and (\ref{statement, main theorem, LOT second contribution range}) above. The former is essentially just Case 1 and 2 above, and so we will only need to provide a bound for Case 2 (since Case 1 contributes zero). The latter is somewhat difficult to evaluate asymptotically, but it can be bounded. Step 2 below will address the main term for this case, which comes from (\ref{statement, main theorem, main term contribution range}). \\

\underline{Step 1:} \\

Suppose $\strictrho (\mathbfe{\alpha}) = 0$. Similarly as in the previous cases, we have that $\mathbfe{\alpha} \in \strictmathcal{L}_n^h (r_1 , 0 , r_1)$ for some $2 \leq r_1 \leq n_1$. Case 1 tells us that if $r_1 \leq s'+1$, then the contribution is zero. By a slight adaptation of Case 2, we can see that the contribution of $s' + 2 \leq r_1 \leq n_1$ is 
\begin{align*}
q^{2h - (n_2 -1)} f_{U,V} (n , n_1 -1) 
\leq  4 q^{2h - n_2 + \frac{3}{2}} 
	\lvert UV \rvert^{\frac{1}{2}} 
	(\log_q \degree U)
	(\log_q \degree V).
\end{align*}

So, we have bounded the contribution from the cases where $\strictrho (\mathbfe{\alpha}) = 0$. Now let us bound the contribution from the cases where (\ref{statement, main theorem, LOT second contribution range}) holds. \\

To this end, suppose that $h +1 < \rho (\mathbfe{\alpha} ) \leq n_2 -1$ and $r (\mathbfe{\alpha}) \geq \min \{ s' , t' \} +1$. The contribution of these cases to (\ref{statement, variance as add char after mean square removed}) is
\begin{align}
\begin{split} \label{statement, main theorem proof, bounding min(s',t') < rho (alpha), Cauchy Schwarz}
&\frac{4 q^{2h}}{q^{2n+1}} \hspace{-1.75em} \sum_{\substack{\mathbfe{\alpha} \in \mathcal{L}_n^h : \\ h +1 < \rho (\mathbfe{\alpha} ) \leq n_2 -1 \\ r (\mathbfe{\alpha}) \geq \min \{ s' , t' \} +1 }} \hspace{-0.5em}
	\bigg\lvert \sum_{\substack{E \in \mathcal{M}_{s'} \\ F \in \mathcal{A}_{\leq t'} }}
		\psi \Big( [E]_{s'}^T H_{s'+1 , s'+1} (\mathbfe{\alpha}\odot [U]_s ) [E]_{s'} 
			+ [F]_{t'}^T H_{t'+1 , t'+1} (\mathbfe{\alpha} \odot [V]_t ) [F]_{t'} \Big) \bigg\rvert^2 \\
\leq &\frac{4 q^{2h}}{q^{2n+1}} 
	\bigg[ \sum_{\substack{\mathbfe{\alpha} \in \mathcal{L}_n^h : \\ r (\mathbfe{\alpha}) \geq \min \{ s' , t' \} +1 }}
		\bigg\lvert \sum_{E \in \mathcal{M}_{s'}}
			\psi \Big( [E]_{s'}^T H_{s'+1 , s'+1} (\mathbfe{\alpha}\odot [U]_s ) [E]_{s'} \Big) \bigg\rvert^4 \bigg]^{\frac{1}{2}} \\
	&\hspace{6em} \times \bigg[ \sum_{\substack{\mathbfe{\alpha} \in \mathcal{L}_n^h : \\ r (\mathbfe{\alpha}) \geq \min \{ s' , t' \} +1 }}
		\bigg\lvert \sum_{F \in \mathcal{A}_{\leq t'}}
			\psi \Big( [F]_{t'}^T H_{t'+1 , t'+1} (\mathbfe{\alpha} \odot [V]_t ) [F]_{t'} \Big) \bigg\rvert^4 \bigg]^{\frac{1}{2}} \\
\leq &\frac{4 q^{2h}}{q^{2n+1}} 
	\bigg[ q^s \sum_{\substack{\mathbfe{\beta} \in \mathcal{L}_{n-s}^{h-s} : \\ r (\mathbfe{\beta}) \geq \min \{ s' , t' \} +1 - s }}
		\bigg\lvert \sum_{E \in \mathcal{M}_{s'}}
			\psi \Big( [E]_{s'}^T H_{s'+1 , s'+1} (\mathbfe{\beta}) [E]_{s'} \Big) \bigg\rvert^4 \bigg]^{\frac{1}{2}} \\
	&\hspace{6em} \times \bigg[ q^t \sum_{\substack{\mathbfe{\beta} \in \mathcal{L}_{n- t}^{h-t} : \\ r (\mathbfe{\beta}) \geq \min \{ s' , t' \} +1 - t}}
		\bigg\lvert \sum_{F \in \mathcal{A}_{\leq t'}}
			\psi \Big( [F]_{t'}^T H_{t'+1 , t'+1} (\mathbfe{\beta}) [F]_{t'} \Big) \bigg\rvert^4 \bigg]^{\frac{1}{2}} .
\end{split}
\end{align}
The first relation uses the Cauchy-Schwarz inequality; and it also removes the condition $h +1 < \rho (\mathbfe{\alpha} ) \leq n_2 -1$ from the sum, which we can do because we are considering upper bounds of a sum of positive values. The second uses the fact that the map $\mathbfe{\alpha} \to \mathbfe{\alpha} \odot [U]_s$ is a linear map from $\mathcal{L}_n^h$ to $\mathcal{L}_{n-s}^{h-s}$, with the kernel of this linear map having $q^s$ elements in it. We also used the fact that $r (\mathbfe{\alpha} \odot [U]_s ) \geq \min \{ s' , t' \} +1 - s$, which requires more justification. For this, we note the following implications
\begin{align*}
&r (\mathbfe{\alpha} \odot [U]_s ) \geq \min \{ s' , t' \} +1 - s \\
\iff &\rank H_{s'+1 , s'+1} (\mathbfe{\alpha} \odot [U]_s ) \geq \min \{ s' , t' \} +1 - s \\
\iff &\dim \kernel H_{s'+1 , s'+1} (\mathbfe{\alpha} \odot [U]_s ) \leq \max \{ 0 , s' - t' \} + s \\
\iff &\dim \kernel H_{s'+1 , n-s'+1} (\mathbfe{\alpha}) T_{n-s'+1 , s'+1} ([U]_s) \leq \max \{ 0 , s' - t' \} + s \\
\Longleftarrow \hspace{0.25em} &\dim \kernel H_{s'+1 , n-s'+1} (\mathbfe{\alpha}) \leq \max \{ 0 , s' - t' \} + s +1 \\
\iff & \rank H_{s'+1 , n-s'+1} (\mathbfe{\alpha}) \geq \min \{ s' , t' \} +1 .
\end{align*}
The first implication follows by definition of $r (\mathbfe{\alpha} \odot [U]_s )$. The second and last simply use the fact that the dimension of the kernel is the number of columns minus the rank. The third follows from Remark \ref{remark, circulant Toeplitz multiplied by vector and by Hankel matrix}. The fourth follows from the fact that $T_{n-s'+1 , s'+1} ([U]_s)$ defines an injective map. Now, the last line is true, which follows from the fact that $r (\mathbfe{\alpha}) \geq \min \{ s' , t' \} +1$ and Remark \ref{remark, rank of Hankel matrix is minimum of rows, columns, and r}. Thus, the first line is true, as required. A similar justification applies for the sum involving $V$ in (\ref{statement, main theorem proof, bounding min(s',t') < rho (alpha), Cauchy Schwarz}). \\

Continuing from (\ref{statement, main theorem proof, bounding min(s',t') < rho (alpha), Cauchy Schwarz}), let us consider the sums over $E$ and $F$ separately. We have that 
\begin{align*}
&\sum_{\substack{\mathbfe{\beta} \in \mathcal{L}_{n-s}^{h-s} : \\ r (\mathbfe{\beta}) \geq \min \{ s' , t' \} +1 - s }}
		\bigg\lvert \sum_{E \in \mathcal{M}_{s'}}
			\psi \Big( [E]_{s'}^T H_{s'+1 , s'+1} (\mathbfe{\beta}) [E]_{s'} \Big) \bigg\rvert^4 \\
= &\sum_{r= \min \{ s' , t' \} +1-s}^{s'}
		\sum_{\mathbfe{\beta} \in \strictmathcal{L}_{n-s}^{h-s} (r,r,0) }
			\bigg\lvert \sum_{E \in \mathcal{M}_{s'}}
				\psi \Big( [E]_{s'}^T H_{s'+1 , s'+1} (\mathbfe{\beta}) [E]_{s'} \Big) \bigg\rvert^4 \\
	&+\sum_{r= \min \{ s' , t' \} +1-s}^{s'+1}
		\sum_{\mathbfe{\beta} \in \strictmathcal{L}_{n-s}^{h-s} (r,r-1,1) }
			\bigg\lvert \sum_{E \in \mathcal{M}_{s'}}
				\psi \Big( [E]_{s'}^T H_{s'+1 , s'+1} (\mathbfe{\beta}) [E]_{s'} \Big) \bigg\rvert^4 \\
= &\sum_{r= \min \{ s' , t' \} +1-s}^{s'}
		\sum_{\mathbfe{\beta} \in \strictmathcal{L}_{n-s}^{h-s} (r,r,0) }
			q^{2 (2s' -r)}
	+\sum_{r= \min \{ s' , t' \} +1-s}^{s'+1}
		\sum_{\mathbfe{\beta} \in \strictmathcal{L}_{n-s}^{h-s} (r,r-1,1) }
			q^{2 (2s' +1 -r)} ,
\end{align*}
where both equalities use the second result in Lemma \ref{lemma, quadratic form values over monics and nonmonics}. Note that the first sum over $r$ does not include $r=s'+1$, because $\strictmathcal{L}_{n-s}^{h-s} (s'+1,s'+1,0) = \strictmathcal{L}_{n-s}^{h-s} \big( \frac{n-s}{2}+1,\frac{n-s}{2}+1,0 \big)$ is empty, as mentioned after Definition \ref{definition, notation for sets in terms of n,h,rho,pi,r}. Now let us use Lemma \ref{lemma, number of sequences of given rhopi form} to count the number of $\mathbfe{\beta}$. We have,
\begin{align}
\begin{split} \label{statement, main theorem proof, bounding min(s',t') < rho (alpha), after Cauchy Schwarz, E sum}
&\sum_{\substack{\mathbfe{\beta} \in \mathcal{L}_{n-s}^{h-s} : \\ r (\mathbfe{\beta}) \geq \min \{ s' , t' \} +1 - s }}
		\bigg\lvert \sum_{E \in \mathcal{M}_{s'}}
			\psi \Big( [E]_{s'}^T H_{s'+1 , s'+1} (\mathbfe{\beta}) [E]_{s'} \Big) \bigg\rvert^4 \\
= &\sum_{r= \min \{ s' , t' \} +1-s}^{s'} (q-1) q^{2r -h+s -1} q^{2 (2s' -r)}
	+ \sum_{r= \min \{ s' , t' \} +1-s}^{s'+1} (q-1)^2 q^{2r-h+s-3} q^{2 (2s' +1 -r)} \\
\leq & q^{2n-h-s+1} \big( s+1 - \max \{ 0 , s'-t' \} \big) .
\end{split}
\end{align}
We apply a similar approach to the sum over $F$ in (\ref{statement, main theorem proof, bounding min(s',t') < rho (alpha), Cauchy Schwarz}), although since $F$ is not restricted to monics, we use the first part of Lemma \ref{lemma, quadratic form values over monics and nonmonics} instead of the second. We obtain
\begin{align}
\begin{split} \label{statement, main theorem proof, bounding min(s',t') < rho (alpha), after Cauchy Schwarz, F sum}
&\sum_{\substack{\mathbfe{\beta} \in \mathcal{L}_{n- t}^{h-t} : \\ r (\mathbfe{\beta}) \geq \min \{ s' , t' \} +1 - t}}
		\bigg\lvert \sum_{F \in \mathcal{A}_{\leq t'}}
			\psi \Big( [F]_{t'}^T H_{t'+1 , t'+1} (\mathbfe{\beta}) [F]_{t'} \Big) \bigg\rvert^4 \\
= &\sum_{r= \min \{ s' , t' \} +1-t}^{t' +1}
	\sum_{\mathbfe{\beta} \in \mathcal{L}_{n- t}^{h-t} (r)}
		\bigg\lvert \sum_{F \in \mathcal{A}_{\leq t'}}
			\psi \Big( [F]_{t'}^T H_{t'+1 , t'+1} (\mathbfe{\beta}) [F]_{t'} \Big) \bigg\rvert^4 \\
= &\sum_{r= \min \{ s' , t' \} +1-t}^{t' +1}
	\sum_{\mathbfe{\beta} \in \mathcal{L}_{n- t}^{h-t} (r)} q^{2(2t' + 2 -r)} \\
= &\bigg( \sum_{r= \min \{ s' , t' \} +1-t}^{t'}
	(q^2 -1) q^{2r-h+t-2} q^{2(2t' + 2 -r)} \bigg)
	+ (q-1) q^{n-h} q^{2t' + 2} \\
\leq &q^{2n-h-t+4} \big( t+1 - \max \{ 0 , t'-s' \} \big) .
\end{split}
\end{align}
Finally, applying (\ref{statement, main theorem proof, bounding min(s',t') < rho (alpha), after Cauchy Schwarz, E sum}) and (\ref{statement, main theorem proof, bounding min(s',t') < rho (alpha), after Cauchy Schwarz, F sum}) to (\ref{statement, main theorem proof, bounding min(s',t') < rho (alpha), Cauchy Schwarz}) gives
\begin{align*}
&\frac{4 q^{2h}}{q^{2n+1}} \hspace{-1.75em} \sum_{\substack{\mathbfe{\alpha} \in \mathcal{L}_n^h : \\ r (\mathbfe{\alpha}) \geq \min \{ s' , t' \} +1 }} \hspace{-0.5em}
	\bigg\lvert \sum_{\substack{E \in \mathcal{M}_{s'} \\ F \in \mathcal{A}_{\leq t'} }}
		\psi \Big( [E]_{s'}^T H_{s'+1 , s'+1} (\mathbfe{\alpha}\odot [U]_s ) [E]_{s'} 
			+ [F]_{t'}^T H_{t'+1 , t'+1} (\mathbfe{\alpha} \odot [V]_t ) [F]_{t'} \Big) \bigg\rvert^2 \\
\ll & q^{h+\frac{3}{2}} \max \{ \degree U , \degree V \} ,
\end{align*}
as required. \\

\underline{Step 2:} \\

We will now consider the contribution to (\ref{statement, variance as add char after mean square removed}) from the range (\ref{statement, main theorem, main term contribution range}), which will give us the main term. We wish to evaluate
\begin{align*}
\frac{4 q^{2h}}{q^{2n+1}} \hspace{-1.75em} \sum_{\substack{\mathbfe{\alpha} \in \mathcal{L}_n^h : \\ h+1 \leq \strictrho (\mathbfe{\alpha} ) \leq \min \{ s' , t' \} \\ r (\mathbfe{\alpha}) \leq \min \{ s' , t' \} \\ \strictpi (\mathbfe{\alpha} \odot [U]_s) \leq 1 }}
	\bigg\lvert \sum_{\substack{E \in \mathcal{M}_{s'} \\ F \in \mathcal{A}_{\leq t'} }}
		\psi \Big( [E]_{s'}^T H_{s'+1 , s'+1} (\mathbfe{\alpha}\odot [U]_s ) [E]_{s'} 
			+ [F]_{t'}^T H_{t'+1 , t'+1} (\mathbfe{\alpha} \odot [V]_t ) [F]_{t'} \Big) \bigg\rvert^2 .
\end{align*}
As before, the condition $\strictpi (\mathbfe{\alpha} \odot [U]_s) \leq 1$ follows from the second result in Lemma \ref{lemma, quadratic form values over monics and nonmonics}. Although, because $r (\mathbfe{\alpha} ) \leq \min \{ s' , t' \}$ and $\degree U = s$, Lemma \ref{lemma, U reduction, 1 dimensional kernel case} tells us that\break $\strictpi (\mathbfe{\alpha} \odot [U]_s) = \strictpi (\mathbfe{\alpha})$. That is, we can impose the additional condition $\strictpi (\mathbfe{\alpha}) \leq 1$ without affecting the final result. Furthermore, note that the two conditions $h+1 \leq \strictrho (\mathbfe{\alpha} ) \leq \min \{ s' , t' \}$ and $\strictpi (\mathbfe{\alpha}) \leq 1$ imply the condition $r (\mathbfe{\alpha}) \leq \min \{ s' , t' \} +1$. Thus, we could remove the third condition $r (\mathbfe{\alpha}) \leq \min \{ s' , t' \}$, were it not for the additional case where $r (\mathbfe{\alpha}) = \min \{ s' , t' \} +1$. However, the contribution of $r (\mathbfe{\alpha}) = \min \{ s' , t' \} +1$ is lower order, and has already been included in the bounds we obtained in Step 1. Thus, for convenience, we can include it without affecting the final result. Thus, we wish to evaluate
\begin{align*}
\frac{4 q^{2h}}{q^{2n+1}} \hspace{-1.75em} \sum_{\substack{\mathbfe{\alpha} \in \mathcal{L}_n^h : \\ h+1 \leq \strictrho (\mathbfe{\alpha} ) \leq \min \{ s' , t' \} \\ \strictpi (\mathbfe{\alpha}) \leq 1 \\ \strictpi (\mathbfe{\alpha} \odot [U]_s) \leq 1 }}
	\bigg\lvert \sum_{\substack{E \in \mathcal{M}_{s'} \\ F \in \mathcal{A}_{\leq t'} }}
		\psi \Big( [E]_{s'}^T H_{s'+1 , s'+1} (\mathbfe{\alpha}\odot [U]_s ) [E]_{s'} 
			+ [F]_{t'}^T H_{t'+1 , t'+1} (\mathbfe{\alpha} \odot [V]_t ) [F]_{t'} \Big) \bigg\rvert^2 .
\end{align*}
Applying Lemma \ref{lemma, quadratic form values over monics and nonmonics} gives
\begin{align*}
&\sum_{\substack{\mathbfe{\alpha} \in \mathcal{L}_n^h : \\ h+1 \leq \strictrho (\mathbfe{\alpha} ) \leq \min \{ s' , t' \} \\ \strictpi (\mathbfe{\alpha}) \leq 1 \\ \strictpi (\mathbfe{\alpha} \odot [U]_s) \leq 1 }}
	\bigg\lvert \sum_{\substack{E \in \mathcal{M}_{s'} \\ F \in \mathcal{A}_{\leq t'} }}
		\psi \Big( [E]_{s'}^T H_{s'+1 , s'+1} (\mathbfe{\alpha}\odot [U]_s ) [E]_{s'} 
			+ [F]_{t'}^T H_{t'+1 , t'+1} (\mathbfe{\alpha} \odot [V]_t ) [F]_{t'} \Big) \bigg\rvert^2 \\
= &\sum_{\substack{\mathbfe{\alpha} \in \mathcal{L}_n^h : \\ h+1 \leq \strictrho (\mathbfe{\alpha} ) \leq \min \{ s' , t' \} \\ \strictpi (\mathbfe{\alpha}) \leq 1 \\ \strictpi (\mathbfe{\alpha} \odot [U]_s) \leq 1 }}
	q^{s'+t' + \strictpi (\mathbfe{\alpha} \odot [U]_s) }
	\Big\lvert \kernel H_{s'+1 , s'+1} (\mathbfe{\alpha}\odot [U]_s ) \Big\rvert
	\Big\lvert \kernel H_{t'+1 , t'+1} (\mathbfe{\alpha}\odot [V]_t ) \Big\rvert .
\end{align*}
Now, almost identically as in Case 2, we can consider $\mathbfe{\alpha}'$ instead of $\mathbfe{\alpha}$, and this simplifies the expression. Ultimately, we obtain
\begin{align*}
&\sum_{\substack{\mathbfe{\alpha} \in \mathcal{L}_n^h : \\ h+1 \leq \strictrho (\mathbfe{\alpha} ) \leq \min \{ s' , t' \} \\ \strictpi (\mathbfe{\alpha}) \leq 1 \\ \strictpi (\mathbfe{\alpha} \odot [U]_s) \leq 1 }}
	\bigg\lvert \sum_{\substack{E \in \mathcal{M}_{s'} \\ F \in \mathcal{A}_{\leq t'} }}
		\psi \Big( [E]_{s'}^T H_{s'+1 , s'+1} (\mathbfe{\alpha}\odot [U]_s ) [E]_{s'} 
			+ [F]_{t'}^T H_{t'+1 , t'+1} (\mathbfe{\alpha} \odot [V]_t ) [F]_{t'} \Big) \bigg\rvert^2 \\
= &q^{s'+t'} \sum_{\substack{\mathbfe{\alpha} \in \mathcal{L}_{n-1}^h : \\ h+1 \leq \strictrho (\mathbfe{\alpha} ) \leq \min \{ s' , t' \} \\ \strictpi (\mathbfe{\alpha}) = 0 \\ \strictpi (\mathbfe{\alpha} \odot [U]_s) = 0 }}
	\Big\lvert \kernel H_{s' , s'+1} (\mathbfe{\alpha} \odot [U]_s ) \Big\rvert
	\Big\lvert \kernel H_{t'+1 , t'+1} (\mathbfe{\alpha} \odot [V]_{t-1} ) \Big\rvert \\
= &q^{s'+t'} \sum_{\substack{\mathbfe{\alpha} \in \mathcal{L}_{n-1}^h : \\ h+1 \leq \strictrho (\mathbfe{\alpha} ) \leq \min \{ s' , t' \} \\ \strictpi (\mathbfe{\alpha}) = 0 }}
	\Big\lvert \kernel H_{s' , s'+1} (\mathbfe{\alpha} \odot [U]_s ) \Big\rvert
	\Big\lvert \kernel H_{t'+1 , t'+1} (\mathbfe{\alpha} \odot [V]_{t-1} ) \Big\rvert \\
= &q^{s'+t'} \sum_{r=h+1}^{ \min \{ s' , t' \} } 
	\sum_{\mathbfe{\alpha} \in \strictmathcal{L}_{n-1}^h (r,r,0) }
	\Big\lvert \kernel H_{s' , s'+1} (\mathbfe{\alpha} \odot [U]_s ) \Big\rvert
	\Big\lvert \kernel H_{t'+1 , t'+1} (\mathbfe{\alpha} \odot [V]_{t-1} ) \Big\rvert ,
\end{align*}
where, for the second equality, we removed the condition $\strictpi (\mathbfe{\alpha} \odot [U]_s) = 0$ because, as before, it follows from the fact that $\strictpi (\mathbfe{\alpha} \odot [U]_s) = \strictpi (\mathbfe{\alpha})$ (which in turn follows from Lemma \ref{lemma, U reduction, 1 dimensional kernel case}). Now, when $\mathbfe{\alpha} \in \strictmathcal{L}_{n-1}^h (r,r,0)$ with $r \leq \min \{ s' , t' \}$, Lemma \ref{lemma, U reduction, 1 dimensional kernel case} tells us that 
\begin{align*}
r (\mathbfe{\alpha} \odot [U]_s)
	= &r - \degree \big( A_1 (\mathbfe{\alpha}) , U \big) , \\
r (\mathbfe{\alpha} \odot [V]_{t-1})
	= &r - \degree \big( A_1 (\mathbfe{\alpha}) , V \big) ;
\end{align*}
and so
\begin{align*}
\Big\lvert \kernel H_{s' , s'+1} (\mathbfe{\alpha} \odot [U]_s ) \Big\rvert
	= &q^{s'+1 - r} \lvert \big( A_1 (\mathbfe{\alpha}) , U \big) \rvert , \\
\Big\lvert \kernel H_{t'+1 , t'+1} (\mathbfe{\alpha} \odot [V]_{t-1} ) \Big\rvert
	= &q^{t'+1 - r} \lvert \big( A_1 (\mathbfe{\alpha}) , V \big) \rvert .
\end{align*}
Hence, 
\begin{align}
\begin{split} \label{statement, main theorem proof, W sum}
&\sum_{\substack{\mathbfe{\alpha} \in \mathcal{L}_n^h : \\ h+1 \leq \strictrho (\mathbfe{\alpha} ) \leq \min \{ s' , t' \} \\ \strictpi (\mathbfe{\alpha}) \leq 1 \\ \strictpi (\mathbfe{\alpha} \odot [U]_s) \leq 1 }}
	\bigg\lvert \sum_{\substack{E \in \mathcal{M}_{s'} \\ F \in \mathcal{A}_{\leq t'} }}
		\psi \Big( [E]_{s'}^T H_{s'+1 , s'+1} (\mathbfe{\alpha}\odot [U]_s ) [E]_{s'} 
			+ [F]_{t'}^T H_{t'+1 , t'+1} (\mathbfe{\alpha} \odot [V]_t ) [F]_{t'} \Big) \bigg\rvert^2 \\
= &q^{2(s'+t' +1)} \sum_{r=h+1}^{ \min \{ s' , t' \} } q^{-2r}
	\sum_{\mathbfe{\alpha} \in \strictmathcal{L}_{n-1}^h (r,r,0) }
	\lvert \big( A_1 (\mathbfe{\alpha}) , U \big) \rvert
	\lvert \big( A_1 (\mathbfe{\alpha}) , V \big) \rvert \\
= &q^{2(s'+t' +1)} \sum_{\substack{U_1 \mid U \\ V_1 \mid V }} \lvert U_1 V_1 \rvert
	\sum_{r=h+1}^{ \min \{ s' , t' \} } q^{-2r}
	\sum_{\substack{\mathbfe{\alpha} \in \strictmathcal{L}_{n-1}^h (r,r,0) \\ \big( A_1 (\mathbfe{\alpha}) , U \big) = U_1 \\ \big( A_1 (\mathbfe{\alpha}) , V \big) = V_1 }} 1 \\
= &q^{2(s'+t' +1)} \sum_{W_1 \mid W} \lvert W_1 \rvert
	\sum_{r=h+1}^{ \min \{ s' , t' \} } q^{-2r}
	\sum_{\substack{\mathbfe{\alpha} \in \strictmathcal{L}_{n-1}^h (r,r,0) \\ \big( A_1 (\mathbfe{\alpha}) , W \big) = W_1 }} 1 \\
= &q^{2(s'+t' +1)} \sum_{W_1 \mid W} \lvert W_1 \rvert
	\sum_{r=h+1}^{ \min \{ s' , t' \} } q^{-2r}
	\sum_{\substack{A \in \mathcal{M}_r \\ B \in \mathcal{A}_{<r-h} \\ (A,B) =1 \\ (A,W) = W_1}} 1.
\end{split}
\end{align}
For the third equality we are writing $W=UV$ and $W_1 = U_1 V_1$, and we have made use of the fact that $U,V$ are coprime. For the final equality we used lemma \ref{lemma, bijection from quasiregular Hankel matrices to M_r times A_<r}. Now, we have that
\begin{align*}
&\sum_{W_1 \mid W} \lvert W_1 \rvert
	\sum_{r=h+1}^{ \min \{ s' , t' \} } q^{-2r}
	\sum_{\substack{A \in \mathcal{M}_r \\ B \in \mathcal{A}_{<r-h} \\ (A,B) =1 \\ (A,W) = W_1}} 1 \\
= &\sum_{\substack{W_2 W_3 = W \\ (W_2 , W_3) =1 }}
	\sum_{W_1 W_1' = W_2} \lvert W_1 \rvert
	\sum_{r=h+1}^{ \min \{ s' , t' \} } q^{-2r}
	\sum_{\substack{A \in \mathcal{M}_r \\ B \in \mathcal{A}_{<r-h} \\ (A,B) =1 \\ (A,W_2) = W_1 \\ \rad (B,W) = \rad W_3}} 1 \\
= &\sum_{\substack{W_2 W_3 = W \\ (W_2 , W_3) =1 }}
	\sum_{W_1 W_1' = W_2} \lvert W_1 \rvert
	\sum_{r=h+1}^{ \min \{ s' , t' \} } q^{-2r}
	\sum_{\substack{B \in \mathcal{A}_{<r-h} \backslash \{ 0 \} \\ \rad (B,W) = \rad W_3}}
	\sum_{\substack{A \in \mathcal{M}_r \\ (A,B) =1 \\ (A,W_2) = W_1 }} 1 . 
\end{align*}
For the first equality, we conditioned on the value of $\rad (B,W)$. Note that the three conditions $(A,B)=1$, $\rad (B,W) = \rad W_3$, and $W_2 W_3 = W$ with $(W_2 , W_3) =1$, imply that $(A,W)$ must divide $W_2$, and this explains the second sum on the second line and the condition $(A,W_2) = W_1$. Now, we have that
\begin{align*}
\sum_{\substack{A \in \mathcal{M}_r \\ (A,B) =1 \\ (A,W_2) = W_1 }} 1
= \sum_{\substack{A \in \mathcal{M}_{r-\degree W_1} \\ (A,B) =1 \\ (A,W_1') = 1 }} 1
= q^{r - \degree W_1} \frac{\phi (B)}{\lvert B \rvert} \frac{\phi (W_1')}{\lvert W_1' \rvert} .
\end{align*}
For this to hold, we are using the fact that $B$ and $W_1'$ are coprime, and we also require that $r - \degree W_1 \geq \degree B + \degree W_1'$, which follows from
\begin{align*}
r > (r-h-1) + h \geq \degree B + h \geq \degree B + \degree UV \geq \degree B + \degree W_1 W_1' .
\end{align*}
Thus,
\begin{align}
\begin{split} \label{statement, main theorem proof, W sum as phi B/B}
&\sum_{W_1 \mid W} \lvert W_1 \rvert
	\sum_{r=h+1}^{ \min \{ s' , t' \} } q^{-2r}
	\sum_{\substack{A \in \mathcal{M}_r \\ B \in \mathcal{A}_{<r-h} \\ (A,B) =1 \\ (A,W) = W_1}} 1 \\
&= \sum_{\substack{W_2 W_3 = W \\ (W_2 , W_3) =1 }} 
	\sum_{W_1 W_1' = W_2} \frac{\phi (W_1')}{\lvert W_1' \rvert}
	\sum_{r=h+1}^{ \min \{ s' , t' \} } q^{-r}
	\sum_{\substack{B \in \mathcal{A}_{<r-h} \backslash \{ 0 \} \\ \rad (B,W) = \rad W_3}}
		\frac{\phi (B)}{\lvert B \rvert} .
\end{split}
\end{align}
We note that
\begin{align*}
&\sum_{r=h+1}^{ \min \{ s' , t' \} } q^{-r}
	\sum_{\substack{B \in \mathcal{A}_{<r-h} \backslash \{ 0 \} \\ \rad (B,W) = \rad W_3}}
	\frac{\phi (B)}{\lvert B \rvert} \\
= &q^{-h-1} \sum_{r=0}^{ \min \{ s' , t' \} -h-1} q^{-r}
	\sum_{\substack{B \in \mathcal{A}_{\leq r} \backslash \{ 0 \} \\ \rad (B,W) = \rad W_3}}
	\frac{\phi (B)}{\lvert B \rvert} \\
= &q^{-h-1} \sum_{\substack{B \in \mathcal{A}_{\leq \min \{ s' , t' \} -h-1} \backslash \{ 0 \} \\ \rad (B,W) = \rad W_3}}
	\bigg( \sum_{r=\degree B}^{ \min \{ s' , t' \} -h-1} q^{-r} \bigg) \frac{\phi (B)}{\lvert B \rvert} \\
= &\frac{q^{-h-1}}{1-q^{-1}} \sum_{\substack{B \in \mathcal{A}_{\leq \min \{ s' , t' \} -h-1} \backslash \{ 0 \} \\ \rad (B,W) = \rad W_3}}
	\frac{\phi (B)}{\lvert B \rvert^2} - q^{h+1-\min \{ s' , t' \} } \frac{\phi (B)}{\lvert B \rvert} \\
= &(1-q^{-1}) q^{-h} \Bigg[
	\Big( \frac{n}{2} -h + O (\degree W) \Big) 
	\prod_{P \mid W} 
		(1 + \lvert P \rvert^{-1})^{-1}
	\prod_{P \mid W_3} 
		\lvert P \rvert^{-1}
	\hspace{1em} + O (1) + O_\epsilon \Big( q^{-\frac{k}{2}} \lvert W \rvert^{\epsilon} \Big) \Bigg] ,
\end{align*}
where we have used Lemma \ref{lemma, phi (B)/B^2 sum, perron application} for the sum involving $\frac{\phi (B)}{\lvert B \rvert^2}$, and we have applied a trivial bound for the sum involving $\frac{\phi (B)}{\lvert B \rvert}$. We apply this to (\ref{statement, main theorem proof, W sum as phi B/B}) and then (\ref{statement, main theorem proof, W sum}), and after some rearrangement, and simplification of the error terms, we obtain
\begin{align*}
&\frac{4 q^{2h}}{q^{2n+1}} \hspace{-2em} \sum_{\substack{\mathbfe{\alpha} \in \mathcal{L}_n^h : \\ h+1 \leq \strictrho (\mathbfe{\alpha} ) \leq \min \{ s' , t' \} \\ \strictpi (\mathbfe{\alpha}) \leq 1 \\ \strictpi (\mathbfe{\alpha} \odot [U]_s) \leq 1 }}
	\bigg\lvert \sum_{\substack{E \in \mathcal{M}_{s'} \\ F \in \mathcal{A}_{\leq t'} }}
		\psi \Big( [E]_{s'}^T H_{s'+1 , s'+1} (\mathbfe{\alpha}\odot [U]_s ) [E]_{s'} 
			+ [F]_{t'}^T H_{t'+1 , t'+1} (\mathbfe{\alpha} \odot [V]_t ) [F]_{t'} \Big) \bigg\rvert^2 \\
= &4 (1-q^{-1}) q^h M (U,V) \Big( \frac{n}{2} -h \Big) 
	+ O_{\epsilon} (q^h \lvert UV \rvert^{-1+\epsilon}) ;
\end{align*}
where, for a non-zero polynomial $A$, we define $e_P (A)$ to be the maximal integer such that $P^{e_P (A)} \mid A$; and
\begin{align*}
M (U,V)
:= \lvert UV \rvert^{-1} 
	\prod_{P \mid UV} \Bigg( 1 + \bigg( \frac{1- \lvert P \rvert^{-1}}{1+ \lvert P \rvert^{-1}} \bigg)  e_P (UV) \Bigg) .
\end{align*}
\end{proof}

\begin{lemma} \label{lemma, phi (B)/B^2 sum, perron application}
Let $W \in \mathcal{M}$, and let $W_1 , W_2 \in \mathcal{M}$ be such that $W = W_1 W_2$ and $(W_1 , W_2)=1$. Let $k$ be a positive integer. Then,
\begin{align*}
&\sum_{\substack{B \in \mathcal{A}_{\leq k} \backslash \{ 0 \} \\ \rad (B,W) = \rad W_3}} 
	\frac{\phi (B)}{\lvert B \rvert^2} \\
&= \frac{(q-1)^2}{q}
	\Big( k + O ( \log \degree W) \Big) 
	\prod_{P \mid W} 
		(1 + \lvert P \rvert^{-1})^{-1}
	\prod_{P \mid W_3} 
		\lvert P \rvert^{-1}
	\hspace{1em} + O_\epsilon \Big( q^{-\frac{k}{2} +1} \lvert W \rvert^{\epsilon} \Big) .
\end{align*}
\end{lemma}

\begin{proof}
Let $c>0$. Using Perron's formula, we have
\begin{align*}
\sum_{\substack{B \in \mathcal{A}_{\leq k} \backslash \{ 0 \} \\ \rad (B,W) = \rad W_3}} 
	\frac{\phi (B)}{\lvert B \rvert^2} 
= &(q-1) \sum_{\substack{B \in \mathcal{M}_{\leq k} \\ \rad W_3 \mid B \\ (B,W_2)=1}} 
	\frac{\phi (B)}{\lvert B \rvert^2} 
= \frac{q-1}{2 \pi i} \int_{\Re (s) =c}
	\frac{q^{(k+\frac{1}{2})s}}{s} 
	\sum_{\substack{B \in \mathcal{M} \\ \rad W_3 \mid B \\ (B,W_2)=1}} 
	\frac{\phi (B)}{\lvert B \rvert^{2+s}} \mathrm{d} s \\
= &\frac{q-1}{2 \pi i} \int_{\Re (s) =c}
	\frac{F(s)}{s} \mathrm{d} s 
\end{align*}
where
\begin{align*}
F(s)
:= q^{(k+\frac{1}{2})s}
	\prod_{P \mid W_3} 
		\bigg( \frac{1 - \lvert P \rvert^{-1}}{\lvert P \rvert^{1+s} -\lvert P \rvert^{-1}} \bigg)
	\prod_{P \mid W_2} 
		\bigg( \frac{\lvert P \rvert^{1+s} -1}{\lvert P \rvert^{1+s} - \lvert P \rvert^{-1}} \bigg)
	\frac{1- q^{-1-s}}{1- q^{-s}} .
\end{align*}
We now shift the contour to $\Re (s) = -\frac{1}{2}$. Formally, this will be done as the limit as $m \longrightarrow \infty$ of a rectangular contour, orientated anticlockwise, with vertices at $c \pm \frac{(4m+1) \pi i}{\log q}$ and $-\frac{1}{2} \pm \frac{(4m+1) \pi i}{\log q}$. The contour itself avoids any singularities. Due to the vertical periodicity of $F(s)$, we have
\begin{align*}
\int_{c + \frac{(4m+1) \pi i}{\log q}}^{-\frac{1}{2} + \frac{(4m+1) \pi i}{\log q}}
	\frac{F(s)}{s} \mathrm{d} s 
= \int_{c + \frac{\pi i}{\log q}}^{-\frac{1}{2} + \frac{\pi i}{\log q}}
	\frac{F(s)}{s + \frac{4m \pi i}{\log q}} \mathrm{d} s 
\ll \frac{\log q}{m} 
	\int_{c + \frac{\pi i}{\log q}}^{-\frac{1}{2} + \frac{\pi i}{\log q}}
	\lvert F(s) \rvert \mathrm{d} s 
\longrightarrow 0 ,
\end{align*}
as $m \longrightarrow \infty$. For the line $\Re (s) = - \frac{1}{2}$, again using the vertical periodicity of $F(s)$, we have
\begin{align*}
&\int_{\Re (s) = - \frac{1}{2}} \frac{F(s)}{s} \mathrm{d} s \\
= &\sum_{m=0}^{\infty}
	\bigg( \int_{t=0}^{1} 
		\frac{F \big(- \frac{1}{2} + \frac{4 \pi i}{\log q} (m+t ) \big)}{- \frac{1}{2} + \frac{4 \pi i}{\log q} (m+t )} \mathrm{d} t
	+ \int_{t=0}^{1} 
		\frac{F \big(- \frac{1}{2} + \frac{4 \pi i}{\log q} (-m-1+t ) \big)}{- \frac{1}{2} + \frac{4 \pi i}{\log q} (-m-1+t )} \mathrm{d} t \bigg) \\
= &\sum_{m=0}^{\infty}
	\bigg( \int_{t=0}^{1} 
		\frac{F \big(- \frac{1}{2} + \frac{4 \pi i}{\log q} t \big)}{- \frac{1}{2} + \frac{4 \pi i}{\log q} (m+t )} \mathrm{d} t
	+ \int_{t=0}^{1} 
		\frac{F \big(- \frac{1}{2} + \frac{4 \pi i}{\log q} t \big)}{- \frac{1}{2} + \frac{4 \pi i}{\log q} (-m-1+t )} \mathrm{d} t \bigg) \\
= &\sum_{m=0}^{\infty}
	\int_{t=0}^{1} 
		\frac{-1 + \frac{4 \pi i}{\log q} (2t-1)}{\Big(- \frac{1}{2} + \frac{4 \pi i}{\log q} (m+t ) \Big) \Big( - \frac{1}{2} + \frac{4 \pi i}{\log q} (-m-1+t ) \Big)} 
		F \Big(- \frac{1}{2} + \frac{4 \pi i}{\log q} t \Big) \mathrm{d} t \\
= &\sum_{m=0}^{\infty} \frac{(\log q)^2}{(m+1)^2}
	\int_{t=0}^{1} 
	\Big\lvert F \Big(- \frac{1}{2} + \frac{4 \pi i}{\log q} t \Big) \Big\rvert \mathrm{d} t \\
= & O_\epsilon \Big( q^{-\frac{k}{2}} \lvert W \rvert^{\epsilon} \Big) ,
\end{align*}
for $\epsilon > 0$. The last line follows from a bound on $\Big\lvert F \Big(- \frac{1}{2} + \frac{4 \pi i}{\log q} t \Big) \Big\rvert$, which can be obtained via similar techniques as those found in Section A.2 of \cite{Yiasemides2020_PhDThesis_PostMinorCorr}. Now let us consider the singularities. $\frac{F(s)}{s}$ has a double pole at $s=0$ and single poles at $s= \pm \frac{2m \pi i}{\log q}$ for integers $m>0$. The single poles cancel due to the fact that
\begin{align*}
\lim_{s \longrightarrow \frac{2m \pi i}{\log q}} \frac{F(s) \big( s - \frac{2m \pi i}{\log q} \big)}{s}
= - \lim_{s \longrightarrow - \frac{2m \pi i}{\log q}} \frac{F(s) \big( s + \frac{2m \pi i}{\log q} \big)}{s} .
\end{align*}
For the double pole we must evaluate
\begin{align*}
(q-1) \lim_{s \longrightarrow 0} \frac{\mathrm{d}}{\mathrm{d}s} s F(s) .
\end{align*}
The product rule must be applied due to the various factors, and the main term (as $k \longrightarrow \infty$) will come from when we differentiate the factor $q^{(k + \frac{1}{2})s}$, while we get lower order terms when we differentiate the other factors (one may wish to use Lemma 4.2.2 of \cite{Yiasemides2020_PhDThesis_PostMinorCorr} for this). We obtain
\begin{align*}
(q-1) \lim_{s \longrightarrow 0} \frac{\mathrm{d}}{\mathrm{d}s} s F(s) 
= &\frac{(q-1)^2}{q}
	\Big( k + O ( \log \degree W) \Big) 
	\prod_{P \mid W_3} 
		\bigg( \frac{1 - \lvert P \rvert^{-1}}{\lvert P \rvert -\lvert P \rvert^{-1}} \bigg)
	\prod_{P \mid W_2} 
		\bigg( \frac{\lvert P \rvert -1}{\lvert P \rvert - \lvert P \rvert^{-1}} \bigg) \\
= &\frac{(q-1)^2}{q}
	\Big( k + O ( \log \degree W) \Big) 
	\prod_{P \mid W} 
		(1 + \lvert P \rvert^{-1})^{-1}
	\prod_{P \mid W_3} 
		\lvert P \rvert^{-1} .
\end{align*}
\end{proof}

\textbf{Acknowledgements:} This research was conducted during a postdoctoral fellowship at the University of Nottingham funded by the EPSRC research grant ``Modular Symbols and Applications'' (grant number EP/S032460/1). The author is grateful to the research council for this funding, and to the grant holder, Nikolaos Diamantis, for his support. The author is also grateful to Ofir Gorodetsky and Joshua Pimm for helpful conversations and references to related work.

%\addcontentsline{toc}{section}{References}
\bibliography{YiasemidesBibliography1}{}

\begin{thebibliography}{10}
\providecommand{\url}[1]{\texttt{#1}}
\providecommand{\urlprefix}{URL }

\bibitem{BerryTabor1977_LevelClusterRegularSpectrum}
M.~V. Berry and M.~Tabor.
\newblock Level Clustering in the Regular Spectrum.
\newblock \emph{Proc. R. Soc. A}, \textbf{356} (1977), 375--394.

\bibitem{BleherLebowitz1994_EnergyLevelStatModelQuantSystUnivScalingLatticePoint}
P.~M. Bleher and J.~L. Lebowitz.
\newblock Energy-level Statistics of Model Quantum Systems: Universality and
  Scaling in a Lattice-point Problem.
\newblock \emph{J. Statist. Phys.}, \textbf{74} (1994), 167--217.

\bibitem{BleherLebowitz1994_VarNumbLatticePointRandomNarrowElliptStrip}
P.~M. Bleher and J.~L. Lebowitz.
\newblock Variance of Number of Lattice Points in Random Narrow Elliptic Strip.
\newblock \emph{Ann. Inst. Henri Poincar\'{e}}, \textbf{31} (1995), 27--58.

\bibitem{Garcia-ArmasGhorpadeRam2011_RelativePrimePolyNonsingHankelMatrFinField}
M.~Garc\'{i}a-Armas, S.~R. Ghorpade and S.~Ram.
\newblock Relatively Prime Polynomials and Nonsingular Hankel Matrices over
  Finite Fields.
\newblock \emph{J. Comb. Theory Ser. A}, \textbf{118} (2011), 819--828.

\bibitem{Heath-Brown1992_DistrMomErrTermDirichletDivProb}
D.~R. Heath-Brown.
\newblock The Distribution and Moments of the Error Term in the Dirichlet
  Divisor Problem.
\newblock \emph{Acta Arith.}, \textbf{60}(4) (1992), 389--415.

\bibitem{HeinigRost1984_AlgMethToeplitzMatrOperat}
G.~Heinig and K.~Rost.
\newblock \emph{Algebraic Methods for Toeplitz-like Matrices and Operators}.
\newblock Birkh\"{a}user Basel, Basel (1984).

\bibitem{HughesRudnick2004_DistrLatticePointThinAnnuli}
C.~P. Hughes and Z.~Rudnick.
\newblock On the Distribution of Lattice Points in Thin Annuli.
\newblock \emph{Int. Math. Res. Not.}, \textbf{2004}(13) (2004), 637--658.

\bibitem{Ivic2009_DivFuncRZFShortInterv}
A.~Ivi\'{c}.
\newblock On the Divisor Function and the Riemann Zeta-function in Short
  Intervals.
\newblock \emph{Ramanujan J.}, \textbf{19} (2009), 207--224.

\bibitem{KeatingRodgersRudnik2018_SumDivFuncFqtMatrInt}
J.~P. Keating, E.~Roditty-Gershon, B.~Rodgers and Z.~Rudnick.
\newblock Sums of Divisor Functions in $\mathbb{F}_q [t]$ and Matrix Integrals.
\newblock \emph{Math. Z.}, \textbf{288} (2018), 167--198.

\bibitem{Wigman2006_DistrLattPointElliptAnnuli}
I.~Wigman.
\newblock The Distribution of Lattice Points in Elliptic Annuli.
\newblock \emph{Quart. J. Math.}, \textbf{57} (2006), 395--423.

\bibitem{Yiasemides2020_PhDThesis_PostMinorCorr}
M.~Yiasemides.
\newblock \emph{Dirichlet $L$-functions and their Derivatives in Function
  Fields}.
\newblock Ph.D. thesis, University of Exeter (2020).
\newline\urlprefix\url{https://ore.exeter.ac.uk/repository/handle/10871/125276}

\bibitem{Yiasemides2021_VariCorrDivFuncFpTHankelMatr_ArXiv_v2}
M.~Yiasemides.
\newblock The Variance and Correlations of the Divisor Function in
  $\mathbb{F}_q [T]$, and Hankel Matrices.
\newblock \emph{arXiv e-prints} (2021).
\newline\urlprefix\url{https://arxiv.org/abs/2110.05959}

\bibitem{Yiasemides2022_VariSumTwoSquareOverIntervalFqT_I_Arxiv}
M.~Yiasemides.
\newblock The Variance of the Sum of Two Squares over Intervals in
  $\mathbb{F}_q [T]$: I.
\newblock \emph{arXiv e-prints} (2022).
\newline\urlprefix\url{https://arxiv.org/abs/2204.04459}

\end{thebibliography}
\bibliographystyle{YiaseBstNumer1}

\end{document}